\documentclass[reqno]{amsart}
\usepackage{amssymb}
\usepackage[usenames, dvipsnames]{color}
\usepackage{pxfonts}
\usepackage{enumerate}
\usepackage{amsmath}
\usepackage[driverfallback=dvipdfm]{hyperref}{\tiny}
\usepackage{verbatim}

\allowdisplaybreaks[4]

\numberwithin{equation}{section}

\newtheorem{theorem}{Theorem}[section]

\newtheorem{lemma}[theorem]{Lemma}
\newtheorem{prop}[theorem]{Proposition}

\theoremstyle{definition}
\newtheorem{remark}[theorem]{Remark}

\theoremstyle{definition}

\theoremstyle{definition}

\theoremstyle{definition}

\makeatletter
\def\dashint{\operatorname%
{\,\,\text{\bf-}\kern-.98em\DOTSI\intop\ilimits@\!\!}}
\makeatother

\def\\det{\text{det}}

\def\.5{\frac{1}{2}}

\def\cD{\mathcal{D}}

\newcommand{\RN}[1]{%
  \textup{\uppercase\expandafter{\romannumeral#1}}%
}

\newcommand{\Div}{\operatorname{div}}

\newcounter{marnote}

\makeatletter
\@namedef{subjclassname@2020}{%
  \textup{2020} Mathematics Subject Classification}
\makeatother

\begin{document}

% ------------------------------------------------------------------------

\title[Perfect conductivity problem with $p$-Laplacian]{Gradient estimates for the $p$-Laplacian perfect conductivity problem with partially flat and $C^{1,\gamma}$ inclusions}

\author[H. Dong]{Hongjie Dong}
\address[H. Dong]{Division of Applied Mathematics, Brown University, 182 George Street, Providence, RI 02912, United States of America.}
\email{Hongjie\_Dong@brown.edu}
\thanks{H. Dong was partially supported by the NSF under agreement DMS-2350129.}
            
\author[L. Xu]{Longjuan Xu}
\address[L. Xu]{Academy for Multidisciplinary Studies, Capital Normal University, Beijing 100048, China.}
\email{longjuanxu@cnu.edu.cn}
\thanks{L. Xu was partially supported by NSF of China (12301141), and Beijing Municipal Education Commission Science and Technology Project (KM202410028001).}

\begin{abstract}
In this paper, we investigate the gradient estimates for solutions to the perfect conductivity problem with two closely spaced perfect conductors embedded in a homogeneous matrix, modeled by $p$-Laplacian elliptic equations. We first prove that the gradient of the solution remains bounded when the conductors possess partially ``flat" boundaries. This contrasts with the case involving strictly convex inclusions, where the gradient can blow up. Second, for conductors with $C^{1,\gamma}$ boundaries ($\gamma\in(0,1)$), we establish both upper and lower bounds on the gradient, with optimal blow-up rates. Furthermore, we provide precise asymptotic expansions in some special cases.
%thereby extending the results of Dong, Yang, and Zhu [Math. Ann. 390:5005--5051 (2024)] from the $C^2$ case  to the $C^{1,\gamma}$ setting. 
\end{abstract}

\maketitle

\section{Introduction and main results}%\label{intro}
It is well known that, when two inclusions are in close proximity, the electric field may explode in the intervening region, potentially leading to material failure \cite{basl,k1963,k1993}. Over the past two decades, quantitative theoretical studies have sought to quantify the electric field in composite materials, modeled by 
\begin{align}\label{linear-eq}
\begin{cases} 
\Div(a_k(x) Du_\varepsilon) = 0 & \text{in } \cD, \\ 
u = \varphi & \text{on } \partial \cD, 
\end{cases}
\end{align}
where $\cD$ is a bounded $C^2$ domain in $\mathbb R^n$, $n\geq2$, 
\begin{align*}
a_k(x)=
\begin{cases}
k,&\quad x\in\cD_1^\varepsilon\cup\cD_2^\varepsilon,\\
1,&\quad x\in\cD\setminus\overline{\cD_1^\varepsilon\cup\cD_2^\varepsilon},
\end{cases}
\end{align*}
$k\in(0,1)\cup(1,\infty)$, $\cD_1^\varepsilon$ and $\cD_2^\varepsilon$ are two inclusions in $\cD$, and $\varepsilon:=\mbox{dist}(\partial\cD_1^\varepsilon,\partial\cD_2^\varepsilon)$. Inclusions are called perfect conductors when $k\rightarrow\infty$. It has been proved that $|Du_\varepsilon|$ will blow up as $\varepsilon\rightarrow0$ in the perfect conductivity problem, with blow-up rate $\frac{1}{\sqrt\varepsilon}$ when $n=2$, $\frac{1}{\varepsilon|\ln\varepsilon|}$ when $n=3$, and $\frac{1}{\varepsilon}$ when $n\geq4$. See, for instance,  \cite{ABV2015, ACKLY1999, ADKY2007,  AKLLL2007, AKL2005, BLY2009, BT2013, BV2000, DL2019, JK2023, KLY2015, KLY2013, KLY2014, KL2019, L2020, LY2009, MPM1988, P1989, Y2007} and the references therein.

The study of the gradient blow-up phenomena has been extended to nonlinear settings. A typical model for two-phase nonlinear composite materials involves a $p$-Laplacian equation defined on domains containing two convex inclusions. The behavior of the background is assumed to follow a power law:
\begin{equation*}
J=\sigma|E|^{p-2}E,\quad p>1.
\end{equation*}
This relation describes a range of physical phenomena, such as nonlinear dielectrics \cite{tw1994}, where $J$ and $E$ represent current and electric field, respectively, and the deformation theory of plasticity \cite{cs1997,i2008,s1993}, in which $J$ and $E$ denote stress and infinitesimal strain, respectively. In our context, the electric field $E$ is defined as the gradient of the electric potential and $\sigma$ represents conductivity. In the linear case, for instance, this potential corresponds to the solution of \eqref{linear-eq}.

In this paper, we study a composite material comprising a background medium embedded with two perfectly conducting inclusions. To state our results precisely, we first describe the mathematical formulation as follows:
Let $\cD$ be a bounded $C^2$ domain in $\mathbb R^n$, and let $\cD_1^0$ and $\cD_2^0$ be two $C^2$ (or $C^{1,\gamma}$ with $\gamma\in(0,1)$) open sets that $\mbox{diam}(\cD_i^0)>c>0$ and $\mbox{dist}(\cD_1^0\cup\cD_2^0,\partial\cD)>c>0$, and touch at the origin with the
inner normal direction of $\partial\cD_1^0$ being the positive $x_n$-axis. For $\varepsilon>0$, we translate $\cD_1^0$ and $\cD_2^0$ along the $x_n$-axis to obtain
\begin{equation*}
\cD_1^\varepsilon:=\cD_1^0+(0',\varepsilon/2)\quad\mbox{and}\quad\cD_2^\varepsilon:=\cD_2^0-(0',\varepsilon/2).
\end{equation*}
Here, we write $x=(x',x_n)$ with $x'\in\mathbb R^{n-1}$. Denote $\Omega^\varepsilon:=\cD\setminus\overline{\cD_1^\varepsilon\cup\cD_2^\varepsilon}$. Let us consider the perfect conductivity problem with $p$-Laplacian as follows:
\begin{align}\label{p-laplace}
\begin{cases}
-\Div(|Du_\varepsilon|^{p-2}Du_\varepsilon)=0&\quad\mbox{in}~\Omega^\varepsilon,\\
u_\varepsilon=U_i^\varepsilon&\quad\mbox{on}~\overline{\cD_i^\varepsilon},\\
\int_{\partial \cD_i^\varepsilon}|Du_\varepsilon|^{p-2}Du_\varepsilon\cdot\nu=0&\quad i=1,2,\\
u_\varepsilon=\varphi&\quad\mbox{on}~\partial \cD,
\end{cases}
\end{align}
where $\varphi\in C^2(\partial D)$ is given, $\nu=(\nu_1,\dots,\nu_n)$ is the unit outer normal vector on $\partial \cD_1^\varepsilon\cup\partial \cD_2^\varepsilon$, $U_1^\varepsilon$ and $U_2^\varepsilon$ are two constants which can be determined by the third condition in \eqref{p-laplace}. By the maximum principle, we have
\begin{equation}\label{bdd-phi}
\inf_{\partial \cD}\varphi\leq U_1^\varepsilon, U_2^\varepsilon\leq \sup_{\partial \cD}\varphi.
\end{equation}
See \cite{dyz2024}.
Here and throughout the paper, for $x\in\partial\cD_i^\varepsilon$, $i=1,2$, we use the notation
\begin{equation*}
Du_\varepsilon\cdot\nu(x):=\lim_{t\rightarrow 0^+}\frac{u_\varepsilon(x+t\nu(x))-u_\varepsilon(x)}{t}\quad\mbox{and}\quad Du_\varepsilon(x):=(Du_\varepsilon\cdot\nu(x))\nu(x).
\end{equation*}
As showed in \cite{dyz2024}, the solution $u_\varepsilon\in W^{1,p}(\cD)$ can be viewed as the unique function which has the minimal energy in appropriate function space:
\begin{align}\label{mini-sol}
\begin{cases}
I_p[u_\varepsilon]=\min_{v\in\mathcal{A}^\varepsilon}I_p[v],\\
I_p[v]:=\int_{\cD}|Dv|^{p},\quad v\in\mathcal{A}^\varepsilon,\\
\mathcal{A}^\varepsilon:=\{v\in W^{1,p}(\cD): Dv=0~\mbox{in}~\cD_1^\varepsilon\cup\cD_2^\varepsilon,~ v=\varphi~\mbox{on}~\partial\cD\}.
\end{cases}
\end{align}
By taking an arbitrary function $v\in W^{1,p}(\cD)$ such that $v=\varphi$ on $\partial\cD$, and $Dv=0$ in $\Omega\subset \cD$, where $\Omega$ is an open set containing $\cup_{0\leq\varepsilon\leq c_0}\overline{\cD_1^\varepsilon\cup\cD_2^\varepsilon}$ with $0<c_0<\mbox{dist}(\cD_1^0\cup\cD_2^0,\partial \cD)$. Then $v\in\mathcal{A}^\varepsilon$ and 
\begin{equation*}
\int_{\cD}|Du_{\varepsilon}|^{p}\leq\int_{\cD}|Dv|^{p}.
\end{equation*}
This in combination with Poincar\'{e}'s inequality and \eqref{bdd-phi} yields that $u_{\varepsilon}\in W^{1,p}(\cD)$ is bounded uniformly in $\varepsilon$.

%The case $p=2$ corresponds to the linear case as mentioned above. 
In \cite{gn2012}, Gorb and Novikov studied the perfect conductivity problem for the $p$-Laplacian in the case $p>n$. Subsequently, Ciraolo and Sciammetta \cite{cs2019,cs20192} investigated the anisotropic perfect conductivity problem and extended the results in \cite{gn2012} to the range $1<p\leq n$. The results in \cite{cs2019,cs20192,gn2012} primarily focus on establishing upper and lower bounds for the gradient. Recently, a more precise characterization of the gradient was derived in \cite{dyz2024}  by capturing its leading order term in the asymptotics expansion.

We would like to mention that all the previous work deal with inclusions with smooth boundaries, say, $C^2$ boundaries. However, from a practical standpoint, one may be interested in understanding how the blow-up rate behaves when the smoothness of the inclusions is reduced as well as under what circumstances blow-up does not occur, that is, when material failure is avoided. In what follows, we will address the two scenarios separately.

\subsection{The case when \texorpdfstring{$\cD_1^{\varepsilon}$}~~and \texorpdfstring{$\cD_2^{\varepsilon}$}~~have partially ``flat" boundaries}

Let $p\in(1,+\infty)$. We assume that $\cD_1^0$ and $\cD_2^0$ have a portion of common boundary on $\{x_n=0\}$, that is, $\partial\cD_1^0\cap\partial\cD_2^0=\Sigma'\subset\mathbb R^{n-1}$ and $|\Sigma'|\neq 0$. Here, $\Sigma'$ is a convex domain with $C^2$ boundary and contains a $(n-1)$-dimensional ball whose center of mass is at the origin, and $|\Sigma'|$ is the area of $\Sigma'$. In this case, we shall show that the minimizing problem \eqref{mini-sol} with $\varepsilon=0$ is equivalent to
\begin{align}\label{varep=0}
\begin{cases}
-\Div(|Du_0|^{p-2}Du_0)=0&\quad\mbox{in}~\Omega^0,\\
u_0=U_0&\quad\mbox{on}~\overline{\cD_1^0}\cup\overline{\cD_2^0},\\
\int_{\partial \cD_1^0\cup\partial\cD_2^0}|Du_0|^{p-2}Du_0\cdot\nu=0,\\
u_0=\varphi&\quad\mbox{on}~\partial \cD,
\end{cases}
\end{align}
where $U_0$ is a constant; see Theorem \ref{equiv-u0} below. The flux along $\partial\cD_1^0$ is defined as follows:
\begin{equation}\label{def-F}
\mathcal{F}:=\int_{\partial \cD_1^0}|Du_0|^{p-2}Du_0\cdot\nu.  
\end{equation}

We denote by $\Gamma_+^\varepsilon$ and $\Gamma_-^\varepsilon$ the part of $\partial\cD_1^\varepsilon$ and $\partial\cD_2^\varepsilon$ near the origin, which can be respectively represented by the graphs of two $C^2$ functions in terms of $x'$:
\begin{equation*}
\Gamma_+^\varepsilon=\left\{x_n=\frac{\varepsilon}{2}+h_1(x'),~|x'|<1\right\},\quad\Gamma_-^\varepsilon=\left\{x_n=-\frac{\varepsilon}{2}+h_2(x'),~|x'|<1\right\},
\end{equation*}
where $h_1$ and $h_2$ are two $C^2$ functions on $B'_{2R}$ satisfying 
\begin{equation}\label{assump-h}
\begin{aligned}
&h_1(x')=h_2(x')=0,\quad x'\in\Sigma',\\
&\|h_1\|_{C^2(B_{2R}')}+\|h_2\|_{C^2(B_{2R}')}\leq c_1,\\
&|\nabla_{x'}h_1(x')|,\quad|\nabla_{x'}h_2(x')|\leq c_1d(x'),\quad x'\in B_{2R}'\setminus\overline{\Sigma'},\\
&c_2d(x')^2\leq h_1(x')-h_2(x')\leq c_1d(x')^2,\quad x'\in B_{2R}'\setminus\overline{\Sigma'},
\end{aligned}
\end{equation}
where $R>0$ is a constant independently of $\varepsilon$,  $c_i$ with $i=1,2$ are some positive constants, and 
\begin{equation}\label{def-dist}
d(x'):=\mbox{dist}(x',\Sigma').
\end{equation}
For $x_0\in\Omega^\varepsilon$ and $0<r<1-|x'_0|$, we denote the narrow region between $\cD_1^\varepsilon$ and $\cD_2^\varepsilon$ by
\begin{equation}\label{def-Omega}
\Omega_{r}^\varepsilon(x_0):=\left\{(x',x_n)\in\Omega^\varepsilon: -\frac{\varepsilon}{2}+h_2(x')<x_n<\frac{\varepsilon}{2}+h_1(x'), ~|x'-x'_0|<r\right\}.
\end{equation}
When $x_0=0$, we set $\Omega_{r}^\varepsilon:=\Omega_{r}^\varepsilon(0)$. Denote 
\begin{equation}\label{def-delta}
\delta(x'):=\varepsilon+h_1(x')-h_2(x'),
\end{equation}
and
\begin{equation*}%\label{def-Gamma}
\Gamma_{+,r}^\varepsilon:=\Gamma_{+}^\varepsilon\cap\overline{\Omega_r^\varepsilon},\quad \Gamma_{-,r}^\varepsilon:=\Gamma_{-}^\varepsilon\cap\overline{\Omega_r^\varepsilon}.
\end{equation*}

Define
\begin{align}\label{def-Theta}
\Theta(\varepsilon;p):=
\frac{\varepsilon}{|\Sigma'|^{\frac{1}{p-1}}}.
\end{align}
In Lemma \ref{lem-Theta}, we shall prove that for any $r\in(0,1)$,
$$\lim_{\varepsilon\rightarrow 0}\int_{\{d(x')<r\}}\left(\frac{\Theta(\varepsilon;p)}{\delta(x')}\right)^{p-1}\ dx'=1,$$ which is a key ingredient in the proof of our desired result as given below.

\begin{theorem}\label{thm-flat}
Let $h_1$, $h_2$ be two $C^2$ functions satisfying \eqref{assump-h}, $p>1$, $n\geq 2$, and let $u_\varepsilon\in W^{1,p}(\cD)$ be a weak solution of \eqref{p-laplace}. Then for small $\varepsilon\in(0,100^{-2})$, $|\Sigma'|>0$, and $x\in\Omega_{R}^\varepsilon$, we have 
\begin{align*}
Du_{\varepsilon}=\left(0',\frac{\Theta(\varepsilon;p)}{\delta(x')}\big(\text{sgn}(\mathcal{F})|\mathcal{F}|^{\frac{1}{p-1}}+f_1(\varepsilon)\big)\right)
+{\bf f}_1(x,\varepsilon),
\end{align*}
where $f_{1}: \mathbb R\rightarrow\mathbb R$ is a function of $\varepsilon$ and ${\bf f}_{1}: \mathbb R^n\times\mathbb R\rightarrow\mathbb R^n$ is a function of $x$ and $\varepsilon$, such that
$$\lim_{\varepsilon\rightarrow0}f_1(\varepsilon)=0,$$
$$|{\bf f}_1(x,\varepsilon)|\leq C_1\left(\frac{\Theta(\varepsilon;p)}{\delta(x')^{1-\beta/2}}\big(|\mathcal{F}|^{\frac{1}{p-1}}+|f_1(\varepsilon)|\big)+e^{-\frac{C_2}{\sqrt\varepsilon+d(x')}}\|\varphi\|_{L^\infty(\partial\cD)}\right),$$
$\delta(x')$ and $\mathcal{F}$ are defined in \eqref{def-delta} and \eqref{def-F}, respectively, 
$\beta\in(0,1)$, and $C_1, C_2>0$ are constants depending on $n,p,c_1$, and $c_2$.
\end{theorem}

Since $|\Sigma'|>0$, one can infer from the result in Theorem \ref{thm-flat} that $|Du_\varepsilon|$ is bounded and thus, there is no blowup occur. This reveals the mathematical mechanism behind the absence of field concentration phenomena as the curvature of the inclusion is zero locally, which is in accordance with the linear case \cite[Theorem 1.4]{clx2021}.

\subsection{The case when \texorpdfstring{$\cD_1^{\varepsilon}$}~~ and \texorpdfstring{$\cD_2^{\varepsilon}$}~~ have \texorpdfstring{$C^{1,\gamma}$} ~~boundaries}

In this case, $h_1$ and $h_2$ are two bounded $C^{1,\gamma}$ functions and \eqref{assump-h} is replaced by
\begin{equation}\label{assump-h-gamma}
\begin{aligned}
&h_1(0')=h_2(0')=0,\quad \nabla_{x'}h_1(0')=\nabla_{x'}h_2(0')=0,\\
&\|h_1\|_{C^{1,\gamma}(B_{2R}')}+\|h_2\|_{C^{1,\gamma}(B_{2R}')}\leq c_3,\\
&|\nabla_{x'}h_1(x')|,~ |\nabla_{x'}h_2(x')|\leq c_3|x'|^{\gamma},\quad 0<|x'|<1,\\
&c_4|x'|^{1+\gamma}\leq h_1(x')-h_2(x'), \quad 0<|x'|<1,
\end{aligned}
\end{equation}
where $c_3$ and $c_4$ are positive constants independently of $\varepsilon$. 
The minimizing problem \eqref{mini-sol} with $\varepsilon=0$ is equivalent to 
\begin{align*}%\label{varep=0-1}
\begin{cases}
-\Div(|Du_0|^{p-2}Du_0)=0&\quad\mbox{in}~\Omega^0,\\
u_0=U_0&\quad\mbox{on}~\overline{\cD_1^0}\cup\overline{\cD_2^0},\\
\int_{\partial \cD_1^0\cup\partial\cD_2^0}|Du_0|^{p-2}Du_0\cdot\nu=0,\\
u_0=\varphi&\quad\mbox{on}~\partial \cD
\end{cases}
\end{align*}
for a constant $U_0$ when $p\geq\frac{n+\gamma}{1+\gamma}$, and is equivalent to
\begin{align}\label{varep=0-2}
\begin{cases}
-\Div(|Du_0|^{p-2}Du_0)=0&\quad\mbox{in}~\Omega^0,\\
u_0=U_i&\quad\mbox{on}~\overline{\cD_i^0}\setminus\{0\},\\
\int_{\partial \cD_i^0}|Du_0|^{p-2}Du_0\cdot\nu=0&\quad i=1,2,\\
u_0=\varphi&\quad\mbox{on}~\partial \cD
\end{cases}
\end{align}
for constants $U_1$ and $U_2$ when $1<p<\frac{n+\gamma}{1+\gamma}$. We refer the reader to \cite[Theorem 2.4]{dyz2024} for the proof of the equivalence when $\partial\cD_i^{\varepsilon}\in C^{2}$, $i=1,2$. The argument there can be readily applied to the case when $\partial\cD_i^{\varepsilon}\in C^{1,\gamma}$, $i=1,2$.

By using \eqref{def-delta} and \eqref{assump-h-gamma}, we have 
\begin{equation}\label{delta-asym}
\delta(x')=\varepsilon+a|x'|^{1+\gamma},
\end{equation}
where $a$ is a function of $x'$ such that
\begin{equation}\label{est-a}
\frac{1}{c_0}\leq a\leq c_0
\end{equation}
for some constant $c_0>1$ depending on $c_3$ and $c_4$. 
Denote 
\begin{align}\label{def-K0}
K=
\begin{cases}
\frac{1+\gamma}{2\pi^{\frac{n-1}{2}}}\cdot\frac{\Gamma(\frac{n-1}{2})\Gamma(p-1)}{\Gamma(\frac{n-1}{1+\gamma})\Gamma(p-\frac{n+\gamma}{1+\gamma})},&\quad p>\frac{n+\gamma}{1+\gamma},\\
\frac{1+\gamma}{2\pi^{\frac{n-1}{2}}}\Gamma(\frac{n-1}{2}),&\quad p=\frac{n+\gamma}{1+\gamma},
\end{cases}
\end{align}
where $\Gamma(z):=\int_{0}^{\infty}t^{z-1}e^{-t}\ dz$ is the gamma function for $z>0$,
and
\begin{align}\label{def-Theta-gamma}
\Theta(\varepsilon;p,\gamma):=
\begin{cases}
\varepsilon^{1-\frac{n-1}{(1+\gamma)(p-1)}},    &p>\frac{n+\gamma}{1+\gamma},\\
|\ln\varepsilon|^{-\frac{1}{p-1}}, & p=\frac{n+\gamma}{1+\gamma}.
\end{cases}
\end{align}

\begin{theorem}\label{thm-gamma}
Let $h_1$ and $h_2$ be two bounded $C^{1,\gamma}$ functions satisfying \eqref{assump-h-gamma}, $p>1$, $n\geq 2$, and let $u_\varepsilon\in W^{1,p}(\cD)$ be a weak solution of \eqref{p-laplace}. Then for sufficiently small $\varepsilon>0$ and $x\in\Omega_{R}^\varepsilon$, the following holds:

(1) When $p\geq \frac{n+\gamma}{1+\gamma}$, then we have the upper bound
\begin{align*}
|Du_\varepsilon(x)|\leq C(K|\mathcal{F}|)^{\frac{1}{p-1}}\frac{\Theta(\varepsilon;p,\gamma)}{\delta(x')}(1+\delta(x')^{\beta/2})+C_1e^{-\frac{C_2}{\varepsilon^{\gamma/(1+\gamma)}+|x'|^\gamma}}\|\varphi\|_{L^\infty(\partial\cD)},
\end{align*}
where $C,C_1>0$ are constants depending on $n,p,\gamma,c_3$, and $c_4$, and $C_2>0$ depends on $n,p,c_3$, and $c_4$; For any $x=(0',x_n)\in\Omega_R^\varepsilon$, we also have the lower bound 
\begin{equation*}
|Du_\varepsilon(0',x_n)|\geq\frac{1}{2C_0}(K|\mathcal{F}|)^{\frac{1}{p-1}}\frac{\Theta(\varepsilon;p,\gamma)}{\varepsilon}-C_1e^{-\frac{C_2}{\varepsilon^{\gamma/(1+\gamma)}}}\|\varphi\|_{L^\infty(\partial\cD)},
\end{equation*}
where $C_0>1$ and $C_1>0$ are constants depending on $n,p,\gamma,c_3$, and $c_4$, and $C_2>0$ depends on $n,p,c_3$, and $c_4$;

(2) When $1<p<\frac{n+\gamma}{1+\gamma}$, then we have 
\begin{align*}
Du_\varepsilon(x)=\big(0',\delta(x')^{-1}(U_1-U_2+f_{2}(\varepsilon))\big)+{\bf f}_{2}(x,\varepsilon),
\end{align*}
where $U_1$ and $U_2$ are defined in \eqref{varep=0-2}, $f_{2}: \mathbb R\rightarrow\mathbb R$ is a function of $\varepsilon$ and ${\bf f}_{2}: \mathbb R^n\times\mathbb R\rightarrow\mathbb R^n$ is a function of $x$ and $\varepsilon$ satisfying 
\begin{equation*}
\lim_{\varepsilon\rightarrow0}f_{2}(\varepsilon)=0,
\end{equation*}
and 
\begin{align*}
|{\bf f}_{2}(x,\varepsilon)|\leq C\left(\delta(x')^{\beta/2-1}(|U_1-U_2|+|f_{2}(\varepsilon)|)\right)+C_1e^{-\frac{C_2}{\varepsilon^{\gamma/(1+\gamma)}+|x'|^\gamma}}\|\varphi\|_{L^\infty(\partial\cD)}.
\end{align*}
Here, $\delta(x')$ and $\mathcal{F}$ are defined in \eqref{def-delta} and \eqref{def-F}, respectively, 
$\beta\in(0,1)$, $C_1>0$ is a constant depending on $n,p,\gamma,c_3$, and $c_4$, and $C,C_2>0$ depend on $n,p,c_3$, and $c_4$.
\end{theorem}

In certain special cases, we can derive more precise asymptotic formulae. See Remarks \ref{rmk-U1-U2-x'} and \ref{rmk-U1U2-11} below.

\begin{remark}
When $\cD_1^{\varepsilon}$ and $\cD_2^{\varepsilon}$ are two bounded $C^{1,\gamma}$ domains and have partially ``flat" boundaries, the arguments in the proof of Theorems \ref{thm-flat} and \ref{thm-gamma} remain valid. Specifically, 
let $h_1$, $h_2$ be two bounded $C^{1,\gamma}$ functions satisfying \eqref{assump-h}, $p>1$, $n\geq 2$, and let $u_\varepsilon\in W^{1,p}(\cD)$ be a weak solution of \eqref{p-laplace}. Then for small $\varepsilon\in(0,100^{-2})$, $|\Sigma'|>0$, and $x\in\Omega_{R}^\varepsilon$, we have 
\begin{align*}
Du_{\varepsilon}=\left(0',\frac{\Theta(\varepsilon;p)}{\delta(x')}\big(\text{sgn}(\mathcal{F})|\mathcal{F}|^{\frac{1}{p-1}}+\tilde f(\varepsilon)\big)\right)
+\tilde {\bf f}(x,\varepsilon),
\end{align*}
where $\tilde f: \mathbb R\rightarrow\mathbb R$ is a function of $\varepsilon$ and $\tilde{\bf f}: \mathbb R^n\times\mathbb R\rightarrow\mathbb R^n$ is a function of $x$ and $\varepsilon$, such that
$$\lim_{\varepsilon\rightarrow0}\tilde f(\varepsilon)=0,$$
$$|\tilde{\bf f}(x,\varepsilon)|\leq C_1\left(\frac{\Theta(\varepsilon;p)}{\delta(x')^{1-\beta/2}}\big(|\mathcal{F}|^{\frac{1}{p-1}}+|\tilde f(\varepsilon)|\big)+e^{-\frac{C_2}{\varepsilon^{\gamma/(1+\gamma)}+d(x')^{\gamma}}}\|\varphi\|_{L^\infty(\partial\cD)}\right),$$
$\delta(x')$ and $\mathcal{F}$ are defined in \eqref{def-delta} and \eqref{def-F}, respectively,
$\beta\in(0,1)$, and $C_1, C_2>0$ are constants independently of $\varepsilon$.    
\end{remark}

\begin{remark}
By comparing the results in \cite[Theorem 1.1]{dyz2024}, we list the blow-up rates for the case when the domains with $C^{1,\gamma}$ and $C^2$ boundaries in the table below: 
\begin{center}
\begin{tabular}{|c|c|c|}
\hline
 & $C^{1,\gamma}$& $C^2$ \\
 \hline
 $1<p<\frac{n+1}{2}$ &$\varepsilon^{-1}$& $\varepsilon^{-1}$\\
 \hline
 $p=\frac{n+1}{2}$ &$\varepsilon^{-1}$&  $\varepsilon^{-1}|\ln\varepsilon|^{-\frac{1}{p-1}}$\\
 \hline
$\frac{n+1}{2}<p<\frac{n+\gamma}{1+\gamma}$ &$\varepsilon^{-1}$&$\varepsilon^{-\frac{n-1}{2(p-1)}}$\\
 \hline
 $p=\frac{n+\gamma}{1+\gamma}$&  $\varepsilon^{-1}|\ln\varepsilon|^{-\frac{1}{p-1}}$&$\varepsilon^{-\frac{n-1}{2(p-1)}}$\\
 \hline
$p>\frac{n+\gamma}{1+\gamma}$&$\varepsilon^{-\frac{n-1}{(1+\gamma)(p-1)}}$&$\varepsilon^{-\frac{n-1}{2(p-1)}}$\\
 \hline
\end{tabular} 
\end{center}
The table above shows that, except when $1<p<\frac{n+1}{2}$, the blow-up rates of the gradient with $C^{1,\gamma}$ boundaries exceed those for $C^{2}$ boundaries. This suggests that reduced regularity of the inclusion may result in an amplified blow-up rate. Indeed, for the linear case, Kang and Yun \cite{KY2019} demonstrated that the blow-up rate in the presence of a bow-tie structure is significantly higher than for inclusions with smooth boundaries in two dimensions. 
\end{remark}

\subsection{Main ingredients of the proofs}
The proofs of Theorems \ref{thm-flat} and \ref{thm-gamma} are inspired by the ideas in \cite{dyz2024}. However, new techniques and substantial modifications are required to precisely capture the influence of partially ``flat" and $C^{1,\gamma}$ boundaries on gradient estimates. 

In the case when the boundaries are partially ``flat", we establish the equivalence between the minimizing problem \eqref{mini-sol} with $\varepsilon=0$ and the problem \eqref{varep=0} for all $p>1$--unlike \cite[Theorem 2.4]{dyz2024}, which requires $p\geq(n+1)/2$. This allows us to prove the convergence of 
$$\int_{\widetilde\Gamma_{-,r}^\varepsilon}|Du_\varepsilon|^{p-2}Du_\varepsilon\cdot\nu$$ 
as $\varepsilon\rightarrow0_+$ for any $r>0$ and $p>1$, where $\widetilde\Gamma_{-,r}^\varepsilon$ is defined in \eqref{def-gamma=dist}. It is worth noting that the presence of partially ``flat" boundaries requires the use of $d(x')<r$ in the definition, rather than $|x'|<r$, which distinguishes our case from those in \cite[Proposition 4.2]{dyz2024} and \cite[Proposition 2.1]{gn2012}. Furthermore, we introduce a new quantity $\Theta(\varepsilon;p)$ defined in \eqref{def-Theta}, arising from the proof of Lemma \ref{lem-Theta}. By combining these results, we derive the estimates of $U_1^\varepsilon-U_2^\varepsilon$ in Theorem \ref{thm-U12}, where the influence of $\mathcal{F}$ becomes apparent. %A comparison with the distance function $\delta(x')$ reveals that $|Du_{\varepsilon}|$ remains bounded. %--contrasting with the blow-up behavior in \cite{dyz2024}; see Theorem \ref{thm-flat} for more details. 
For $C^{1,\gamma}$ boundaries case, we also introduce a new quantity $\Theta(\varepsilon;p,\gamma)$ in \eqref{def-Theta-gamma} and establish the upper and lower bounds of $Du_\varepsilon$.

The remainder of the paper is organized as follows. In Section \ref{sec-pre}, we give some preliminary estimates which will be used in the proof of the main results. The proofs of Theorems \ref{thm-flat} and \ref{thm-gamma} are given in Sections \ref{sec-thm1} and \ref{sec-thm2}, respectively.

\section{Preliminaries}\label{sec-pre}

In this section, we present some preliminary results. 

\begin{lemma}\label{lem-Dv}
Let $p>1$, $n\geq 2$, $\varepsilon\in[0,100^{-2})$, and $v\in W^{1,p}(\Omega_1^\varepsilon)$ be a weak solution of 
\begin{align}\label{eq-v-p}
\begin{cases}
-\Div(|Dv|^{p-2}Dv)=0&\quad\mbox{in}~\Omega_1^\varepsilon,\\
v=0&\quad\mbox{on}~\Gamma_{\pm}^\varepsilon.
\end{cases}
\end{align}
For any $x\in\overline{\Omega_{1/3}^\varepsilon}$, the following assertions hold:

(i) when $h_1$ and $h_2$ are two $C^2$ functions satisfying \eqref{assump-h} and \eqref{sigma'} holds for some constant $R_0\in (0,1/3)$, we have
\begin{equation*}
|v(x)|+|Dv(x)|\leq C_1e^{-\frac{C_2}{\sqrt\varepsilon+d(x')}}\|v\|_{L^p(\Omega_{1}^\varepsilon\setminus\Omega_{1/2}^\varepsilon)},
\end{equation*}
where $C_1, C_2>0$ are constants depending on $n,p,c_1$, and $c_2$;

(ii) when $h_1$ and $h_2$ are two $C^{1,\gamma}$ functions satisfying \eqref{assump-h-gamma}, we have
\begin{equation*}
|v(x)|+|Dv(x)|\leq C_1e^{-\frac{C_2}{\varepsilon^{\gamma/(1+\gamma)}+|x'|^\gamma}}\|v\|_{L^p(\Omega_{1}^\varepsilon\setminus\Omega_{1/2}^\varepsilon)},
\end{equation*}
where $C_1>0$ is a constant depending on $n,p,\gamma,c_3$, and $c_4$, and $C_2>0$ depends on $n,p,c_3$, and $c_4$.
\end{lemma}

The proof of Lemma \ref{lem-Dv} builds upon the iterative techniques developed in \cite[Theorem 1.1]{llby2014} and \cite[Lemma 2.1]{dyz2024}, incorporating the assumptions on $h_1$ and $h_2$ specified in \eqref{assump-h} and \eqref{assump-h-gamma}, respectively. It is important to note that the key point is to estimate $v$ and $Dv$ in the narrow region. To this end, one usually considers the distance $\delta(x')$ between the two subdomains $\cD_1$ and $\cD_2$. Under the assumption that $h_1$ and $h_2$ are $C^2$ functions satisfying \eqref{assump-h}, it follows from the definition of $\delta(x')$ that it depends on $d(x')$, which is the distance from $x'$ to $\Sigma'$. This explains the appearance of the term $\sqrt{\varepsilon} + d(x')$. When $h_1$ and $h_2$ are of class $C^{1,\gamma}$ functions and satisfy \eqref{assump-h-gamma}, the iterative procedure requires taking the number of iterations as $k=\lfloor\frac{1}{r^\gamma}\rfloor$ with $r=O(\varepsilon^{1/(1+\gamma)}+|x'|)$, to ensure the validity of the process. For the completeness of the paper and reader's convenience,  we present the details in the Appendix \ref{Append}. 

Replicating the proof of \cite[Proposition 1.5]{dyz2024}, where the assumptions on $h_1$ and $h_2$ are replaced with \eqref{assump-h} and \eqref{assump-h-gamma}, respectively, we derive the following $C^\beta$ bound of the gradient.

\begin{prop}\label{prop-holder}
Let $p>1$, $n\geq 2$, $\varepsilon\in(0,1)$, and $u_\varepsilon\in W^{1,p}(\cD)$ be a weak solution of \eqref{p-laplace}. Then there exists a constant $\beta\in(0,1)$ depending on $n$ and $p$ such that,

(i) if $h_1$ and $h_2$ are two $C^2$ functions satisfying \eqref{assump-h}, then
\begin{equation*}
[Du_\varepsilon]_{C^\beta(\Omega_{x,\sqrt{\tilde\delta(x')}/4}^\varepsilon)}\leq C\tilde\delta(x')^{-\beta/2}\|Du_\varepsilon\|_{L^\infty(\Omega_{x,\sqrt{\tilde\delta(x')}/2}^\varepsilon)},\quad x\in\Omega_{1/4}^\varepsilon,
\end{equation*}
where $\tilde\delta(x')=\varepsilon+d(x')^2$ and $C>0$ is a constant depending on $n,p,c_1$, and $c_2$;

(ii) if $h_1$ and $h_2$ are two $C^{1,\gamma}$ functions satisfying \eqref{assump-h-gamma}, then
\begin{equation*}
[Du_\varepsilon]_{C^\beta(\Omega_{x,\hat\delta(x')^{1/(1+\gamma)}/4}^\varepsilon)}\leq C\hat\delta(x')^{-\beta/2}\|Du_\varepsilon\|_{L^\infty(\Omega_{x,\hat\delta(x')^{1/(1+\gamma)}/2}^\varepsilon)},\quad x\in\Omega_{1/4}^\varepsilon,
\end{equation*}
where $\hat\delta(x')=\varepsilon+|x'|^{1+\gamma}$ and $C>0$ is a constant depending on $n,p,c_3$, and $c_4$.
\end{prop}

With Lemma \ref{lem-Dv} and Proposition \ref{prop-holder} in hand, by mimicking the argument that led to \cite[Proposition 4.1]{dyz2024}, we obtain an asymptotic formula of $Du_\varepsilon$ in terms of $U_1^\varepsilon-U_2^\varepsilon$ as follows.

\begin{prop}
Let $p>1$, $n\geq 2$, $\varepsilon\in[0,100^{-2})$, $U_1^\varepsilon$ and $U_2^\varepsilon$ be arbitrary constants with $|U_i^\varepsilon|\leq \|\varphi\|_{L^\infty(\partial\cD)}$, and $u_\varepsilon\in W^{1,p}(\Omega^\varepsilon)$ be a weak solution of 
\begin{align*}
\begin{cases}
-\Div(|Du_\varepsilon|^{p-2}Du_\varepsilon)=0&\quad\mbox{in}~\Omega^\varepsilon,\\
u_\varepsilon=U_i^\varepsilon&\quad\mbox{on}~\partial\cD_i^\varepsilon,\\
u_\varepsilon=\varphi&\quad\mbox{on}~\partial \cD.
\end{cases}
\end{align*} 
For any $x\in\overline{\Omega_{1/4}^\varepsilon}$, the following assertions hold: 

(i) when $h_1$ and $h_2$ are two $C^2$ functions satisfying \eqref{assump-h},  we have 
\begin{equation}\label{asym-Du}
Du_\varepsilon(x)=\left(0',\frac{U_1^\varepsilon-U_2^\varepsilon}{\delta(x')}\right)+{\bf f}_1(x,\varepsilon),
\end{equation}
where $\delta(x')$ is defined in \eqref{def-delta}, and 
$$|{\bf f}_1(x,\varepsilon)|\leq C_1\left(\frac{|U_1^\varepsilon-U_2^\varepsilon|}{\delta(x')^{1-\beta/2}}+e^{-\frac{C_2}{\sqrt\varepsilon+d(x')}}\|\varphi\|_{L^\infty(\partial\cD)}\right),$$
where $\beta\in(0,1)$, $C_1>0$ and $C_2>0$ depend on $n,p,c_1$, and $c_2$.

(ii) when $h_1$ and $h_2$ are two $C^{1,\gamma}$ functions satisfying \eqref{assump-h-gamma},  we have 
\begin{equation}\label{asym-Du-gamma}
Du_\varepsilon(x)=\left(0',\frac{U_1^\varepsilon-U_2^\varepsilon}{\delta(x')}\right)+{\bf f}_2(x,\varepsilon),
\end{equation}
where 
$$|{\bf f}_2(x,\varepsilon)|\leq C\frac{|U_1^\varepsilon-U_2^\varepsilon|}{\delta(x')^{1-\beta/2}}+C_1e^{-\frac{C_2}{\varepsilon^{\gamma/(1+\gamma)}+|x'|^\gamma}}\|\varphi\|_{L^\infty(\partial\cD)}.$$
Here $\beta\in(0,1)$, $C_1>0$ is a constant depending on $n,p,\gamma,c_3$, and $c_4$, and $C,C_2>0$ depend on $n,p,c_3$, and $c_4$.
\end{prop}

\section{Proof of Theorem \ref{thm-flat}}\label{sec-thm1}

In this section, we give the proof of Theorem \ref{thm-flat}. We first justify the equivalence between the minimizing problem \eqref{mini-sol} with $\varepsilon=0$ and the problem \eqref{varep=0}. %Note that this is different from the relatively convex case considered in \cite{dyz2024}, where an additional condition $p\ge (n+ 1)/2$ in needed.
\begin{theorem}\label{equiv-u0}
Let $p>1$. $u_0$ is the minimizer of \eqref{mini-sol} with $\varepsilon=0$ if and only if $u_0\in W^{1,p}(\cD)$ satisfies \eqref{varep=0}.
\end{theorem}

\begin{proof}
By using the convexity of $I_p$ and $\mathcal{A}^0$, one can see that the minimizer of \eqref{mini-sol} with $\varepsilon=0$ is unique. Let $u^1, u^2\in W^{1,p}(\cD)$ be the solutions of \eqref{varep=0}. Multiplying the equation in \eqref{varep=0} by $u^1-u^2$ and integrating by parts, we have 
\begin{align*}
\int_{\Omega^0}|Du^j|^{p-2}Du^j\cdot D(u^1-u^2)\ dx=0,\quad j=1,2.
\end{align*}
Thus, we have
\begin{align*}
0&=\int_{\Omega^0}\big(|Du^1|^{p-2}Du^1-|Du^2|^{p-2}Du^2\big)\cdot D(u^1-u^2)\ dx\\
&\geq\frac{\min\{1,p-1\}}{2^{p-2}}\int_{\Omega^0}\big(|Du^1|+|Du^2|\big)^{p-2}|Du^1-Du^2|^2\ dx,
\end{align*}
which implies $u^1=u^2$. This means that \eqref{varep=0} has at most one solution $u_0\in W^{1,p}(\cD)$. 

Next, we show that the minimizer $u_0$ of \eqref{mini-sol} with $\varepsilon=0$ satisfies \eqref{varep=0}. For this, we first take $v \in C_c^\infty(\Omega^0)$ in 
\begin{equation*}
0 = \left. \frac{d}{dt} I_p[u_0 + tv] \right|_{t=0}
\end{equation*}
to get 
\begin{equation}\label{varia-u0}
0 = \int_{\Omega^0} |Du_0|^{p-2} Du_0 \cdot Dv \, dx.
\end{equation}
This gives 
$$
-\text{div}(|Du_0|^{p-2} Du_0) = 0 \quad \text{in} \quad \Omega^0.
$$
We next take $v \in C_c^\infty(\cD)$ such that $v = 1 \text{ in } \overline{\cD_1^0} \cup \overline{\cD_2^0}$. Then by using \eqref{varia-u0} and integration by parts, we have 
\begin{align*}
0&= \int_{\Omega^0} |Du_0|^{p-2} Du_0 \cdot Dv \, dx\\
&=-\int_{\Omega^0} \text{div}(|Du_0|^{p-2} Du_0) v \, dx + \sum_{i=1}^2 \int_{\partial D_i^0} |Du_0|^{p-2} Du_0 \cdot \nu \, dS\\
&=\sum_{i=1}^2 \int_{\partial D_i^0} |Du_0|^{p-2} Du_0 \cdot \nu \, dS.
\end{align*}
Therefore, it suffices to prove that $u_0$ equals the same constant on $\overline{\cD_1^0}$ and $\overline{\cD_2^0}$. In fact, from the proof of \cite[Theorem 2.5]{dyz2024}, it follows that, as $\varepsilon\rightarrow0$, 
\begin{equation*}%\label{weak-converg}
u_\varepsilon\rightharpoonup u_0~\mbox{weakly~ in}~W^{1,p}(\cD) \quad\mbox{and}\quad u_\varepsilon\rightarrow u_0~\mbox{in}~L^{p}(\cD),
\end{equation*}
and for some $\beta>0$ and any $K\subset\subset\cD\setminus(\cup_{0<\varepsilon\leq \varepsilon_0}(\cD_1^\varepsilon\cup\cD_2^\varepsilon)\cup\Sigma)$ with $\varepsilon_0>0$, 
\begin{equation}\label{strong-converg}
u_\varepsilon\rightarrow u_0~\mbox{strongly~ in}~C^{1,\beta}(K),
\end{equation}
where $u_\varepsilon$ is the solution of \eqref{p-laplace}, $u_0$ is the minimizer of \eqref{mini-sol} with $\varepsilon=0$, and $\Sigma=\Sigma'\times(-\frac{\varepsilon_0}{2},\frac{\varepsilon_0}{2})$. Thus, we only need to show that $U_1^\varepsilon\rightarrow U_2^\varepsilon$ as $\varepsilon\rightarrow0$. To this end, by the fundamental theorem of calculus, we have for $x'\in\Sigma'$,
\begin{align*}
U_1^\varepsilon-U_2^\varepsilon=\int_{-\frac{\varepsilon}{2}}^{\frac{\varepsilon}{2}}D_nu_\varepsilon(x)\ dx_n.
\end{align*}
Using H\"{o}lder's inequality, we obtain
\begin{align*}
\int_{\Sigma'}|U_1^\varepsilon-U_2^\varepsilon|^p\ dx'\leq C\varepsilon^{p-1}\int_{\Sigma}|Du_\varepsilon|^p\ dx.
\end{align*}
Combining with the fact that $u_{\varepsilon}\in W^{1,p}(\cD)$ is bounded uniformly in $\varepsilon$ (see the argument below \eqref{mini-sol}), as $\varepsilon\rightarrow0$, we get
\begin{equation*}%\label{converg-U1U2}
|U_1^\varepsilon-U_2^\varepsilon|\rightarrow0.
\end{equation*}
This implies that $u_0$ is equal to the same constant in $\overline{\cD_1^0}$ and $\overline{\cD_2^0}$. We complete the proof of this theorem.
\end{proof}

\begin{lemma}\label{lem-Theta}
Let $p>1$, $n\geq 2$, and $r\in(0,1)$. We have
\begin{equation*}
\lim_{\varepsilon\rightarrow0}\int_{\{d(x')<r\}}\left(\frac{\Theta(\varepsilon;p)}{\delta(x')}\right)^{p-1}\ dx'=1,
\end{equation*} 
where $d(x')$, $\delta(x')$, and $\Theta(\varepsilon;p)$ are defined in \eqref{def-dist}, \eqref{def-delta}, and \eqref{def-Theta}, respectively.
\end{lemma}

\begin{proof}
Denote by $R_0$ and $\tilde R_0$ the lengths of the longest and shortest principal semi-axis of $\Sigma'$, respectively.  By a well-known property for bounded convex domains (see, for instance, \cite[Theorem 1.8.2]{g2001}), we have 
\begin{equation}\label{sigma'}
B_{r_0}'\subset\Sigma'\subset B_{R_0}'\quad\mbox{in~a~suitable coordinate system},
\end{equation}
where $r_0=(n-1)^{-3/2}\tilde R_0$. Since $\{|x'|<r\}\subset\{d(x')<r\}\subset\{|x'|<r+R_0\}$, we have 
\begin{equation}\label{int-d-r}
\int_{|x'|<r}\left(\frac{1}{\delta(x')}\right)^{p-1}\ dx'\leq\int_{\{d(x')<r\}}\left(\frac{1}{\delta(x')}\right)^{p-1}\ dx'\leq \int_{|x'|<r+R_0}\left(\frac{1}{\delta(x')}\right)^{p-1}\ dx'.
\end{equation}

It follows from \eqref{def-delta} and \eqref{assump-h} that
\begin{align}\label{int-delta}
\int_{|x'|<r+R_0}\left(\frac{1}{\delta(x')}\right)^{p-1}\ dx'
&=\int_{\Sigma'}\frac{1}{\varepsilon^{p-1}}\ dx'+\int_{\{|x'|<r+R_0\}\setminus\Sigma'}\left(\frac{1}{\varepsilon+h_1(x')-h_2(x')}\right)^{p-1}\ dx'\nonumber\\
&=\frac{|\Sigma'|}{\varepsilon^{p-1}}+\int_{\{|x'|<r+R_0\}\setminus\Sigma'}\left(\frac{1}{\varepsilon+h_1(x')-h_2(x')}\right)^{p-1}\ dx'\nonumber\\
&\leq \frac{|\Sigma'|}{\varepsilon^{p-1}}+C\int_{\{|x'|<r+R_0\}\setminus\Sigma'}\left(\frac{1}{\varepsilon+d(x')^2}\right)^{p-1}\ dx',
\end{align}
where $C>0$ is a constant independently of $\varepsilon$.  If $n=2$, then $\Sigma'=(-R_0,R_0)$ and $d(x')=|x'|-R_0$. Thus,
\begin{align*}
\int_{\{|x'|<r+R_0\}\setminus\Sigma'}\left(\frac{1}{\varepsilon+d(x')^2}\right)^{p-1}\ dx'&=2\int_{R_0}^{r+R_0}\left(\frac{1}{\varepsilon+(s-R_0)^2}\right)^{p-1}\ ds\\
&\leq\mathcal{C}_{1}
\begin{cases}
\varepsilon^{\frac{3}{2}-p},& p>\frac{3}{2},\\
|\ln\varepsilon|,& p=\frac{3}{2},\\
1,& 1<p<\frac{3}{2},
\end{cases}
\end{align*}
where $\mathcal{C}_{1}>0$ is a constant independently of $\varepsilon$. Similarly, we have 
\begin{align*}
\int_{|x'|<r}\left(\frac{1}{\delta(x')}\right)^{p-1}\ dx'
&\geq \frac{|\Sigma'|}{\varepsilon^{p-1}}+\frac{1}{C}\int_{\{|x'|<r\}\setminus\Sigma'}\left(\frac{1}{\varepsilon+d(x')^2}\right)^{p-1}\ dx'\nonumber\\
&\geq\frac{|\Sigma'|}{\varepsilon^{p-1}}+\frac{1}{\mathcal{C}_{1}}\begin{cases}
\varepsilon^{\frac{3}{2}-p},& p>\frac{3}{2},\\
|\ln\varepsilon|,& p=\frac{3}{2},\\
1,& 1<p<\frac{3}{2}.
\end{cases}
\end{align*}
This in combination with \eqref{int-d-r} and \eqref{int-delta} yields 
\begin{align*}
\lim_{\varepsilon\rightarrow0}\frac{\varepsilon^{p-1}}{|\Sigma'|}\int_{\{d(x')<r\}}\left(\frac{1}{\delta(x')}\right)^{p-1}\ dx'=1,
\end{align*}
which is the desired result when $n=2$.

If $n\geq 3$, then by \eqref{sigma'}, we have for small $\varepsilon>0$,
\begin{align*}
&\int_{\{|x'|<r+R_0\}\setminus\Sigma'}\left(\frac{1}{\varepsilon+d(x')^2}\right)^{p-1}\ dx'\\
&\leq\int_{|x'|<r+R_0}\left(\frac{1}{\varepsilon+\mbox{dist}(x',B'_{R_0})^2}\right)^{p-1}\ dx'\\
&\leq C\int_{0}^{R_0}\frac{s^{n-2}}{\varepsilon^{p-1}}\ ds+C\int_{R_0}^{r+R_0}\frac{s^{n-2}}{(\varepsilon+(s-R_0)^{2})^{p-1}}\ ds\\
&\leq \frac{CR_0^{n-1}}{\varepsilon^{p-1}}+C\int_{0}^{r}\frac{(s+R_0)^{n-2}}{(\varepsilon+s^{2})^{p-1}}\ ds\\
&\leq \frac{C|\Sigma'|}{\varepsilon^{p-1}}+C\int_{0}^{r}\frac{s^{n-2}}{(\varepsilon+s^{2})^{p-1}}\ ds+C\int_{0}^{r}\frac{R_0^{n-2}}{(\varepsilon+s^{2})^{p-1}}\ ds.
\end{align*}
We next estimate the last two terms on the right-hand side of the above inequality, respectively. By using a suitable change of variables, for small $\varepsilon>0$, we have 
\begin{align*}
\int_{0}^{r}\frac{s^{n-2}}{(\varepsilon+s^{2})^{p-1}}\ ds&\leq C\varepsilon^{\frac{n+1}{2}-p}\int_{0}^{\frac{r}{\sqrt\varepsilon}}\frac{s^{n-2}}{(1+s^2)^{p-1}}\ ds\\
&\leq C\varepsilon^{\frac{n+1}{2}-p}\left(1+\int_{1}^{\frac{r}{\sqrt\varepsilon}}s^{n-2p}\ ds\right)\leq C
\begin{cases}
\varepsilon^{\frac{n+1}{2}-p}, & p>\frac{n+1}{2},\\
|\ln\varepsilon|, & p=\frac{n+1}{2},\\
1, & 1<p<\frac{n+1}{2}.
\end{cases}
\end{align*}
Similarly, we have 
\begin{align*}
\int_{0}^{r}\frac{R_0^{n-2}}{(\varepsilon+s^{2})^{p-1}}\ ds&\leq CR_0^{n-2}\varepsilon^{\frac{3}{2}-p}\int_{0}^{\frac{r}{\sqrt\varepsilon}}\frac{1}{(1+s^2)^{p-1}}\ ds\\
&\leq CR_0^{n-2}\varepsilon^{\frac{3}{2}-p}\left(1+\int_{1}^{\frac{r}{\sqrt\varepsilon}}\frac{1}{s^{2(p-1)}}\ ds\right)\leq CR_0^{n-2}
\begin{cases}
\varepsilon^{\frac{3}{2}-p}, & p>\frac{3}{2},\\
|\ln\varepsilon|, & p=\frac{3}{2},\\
1, & 1<p<\frac{3}{2}.
\end{cases}
\end{align*}
For small $\varepsilon>0$ such that $\sqrt\varepsilon<R_0$, we have 
$$R_0^{n-2}\varepsilon^{\frac{3}{2}-p}\leq CR_0^{n-1}\varepsilon^{1-p}.$$
By applying Young's inequality, we obtain
\begin{equation*}
CaR_0^{n-2}\leq CR_0^{n-1}a^{\frac{n-1}{n-2}}+C\leq C\big(|\Sigma'|a^{\frac{n-1}{n-2}}+1\big),
\end{equation*}
where $a=|\ln\varepsilon|$ when $p=\frac{3}{2}$, and $a=1$ when $1<p<\frac{3}{2}$. Thus, for small $\varepsilon>0$, we derive
\begin{align*}
\int_{0}^{r}\frac{R_0^{n-2}}{(\varepsilon+s^{2})^{p-1}}\ ds\leq C
\begin{cases}
\frac{|\Sigma'|}{\varepsilon^{p-1}}, & p\geq \frac{3}{2},\\
|\Sigma'|+1, & 1<p<\frac{3}{2}.
\end{cases}
\end{align*}
Therefore, we obtain 
\begin{align}\label{int-n-3}
\int_{\{|x'|<r+R_0\}\setminus\Sigma'}\left(\frac{1}{\varepsilon+d(x')^2}\right)^{p-1}\ dx'\leq \frac{\mathcal{C}_{2}|\Sigma'|}{\varepsilon^{p-1}}+\mathcal{C}_{2}
\begin{cases}
\varepsilon^{\frac{n+1}{2}-p}, & p>\frac{n+1}{2},\\
|\ln\varepsilon|, & p=\frac{n+1}{2},\\
1, & 1<p<\frac{n+1}{2},
\end{cases}
\end{align}
where $\mathcal{C}_{2}>0$ is a constant independently of $\varepsilon$. For any point $x'_0\in\partial\Sigma'$, we define 
$$Q:=\{x'_0+t\nu+tv: 0<t<r_1, v\in\mathbb R^{n-1}, v\perp\nu, |v|<\tan\theta\}$$
and let the axis of $Q$ be the outward normal direction $\nu$ to $\partial\Sigma'$ at $x'_0$. Here, we take $r_1=\min(r_0/2,R)$ and choose a small angle $\theta=\arcsin{\frac{r_0}{4R_0}}>0$. 
Then $Q$ lies outside $\Sigma'$ and $d(x')\leq |x'-x'_0|$ for any $x'\in\{|x'|<r\}\setminus\Sigma'$. For small $\varepsilon>0$, by using \eqref{def-delta} and \eqref{assump-h}, we have
\begin{align}\label{int-lower-3}
\int_{|x'|<r}\left(\frac{1}{\varepsilon+d(x')^2}\right)^{p-1}\ dx'&\geq \frac{|\Sigma'|}{\varepsilon^{p-1}}+\frac{1}{C}\int_{\{|x'|<r\}\setminus\Sigma'}\left(\frac{1}{\varepsilon+d(x')^2}\right)^{p-1}\ dx'\nonumber\\
&\geq \frac{|\Sigma'|}{\varepsilon^{p-1}}+\frac{1}{C}\int_{Q}\left(\frac{1}{\varepsilon+|x'-x'_0|^2}\right)^{p-1}\ dx'\nonumber\\
&=\frac{|\Sigma'|}{\varepsilon^{p-1}}+\frac{|A|}{C}\int_{0}^{r_1}\frac{t^{n-2}}{(\varepsilon+t^2)^{p-1}}\ dt\nonumber\\
&\geq\frac{|\Sigma'|}{\varepsilon^{p-1}}+\frac{1}{\mathcal{C}_{2}}\begin{cases}
\varepsilon^{\frac{n+1}{2}-p}, & p>\frac{n+1}{2},\\
|\ln\varepsilon|, & p=\frac{n+1}{2},\\
1, & 1<p<\frac{n+1}{2},
\end{cases}
\end{align}
where $t=|x'-x'_0|$ and $A\subset\mathbb R^{n-2}$ with $|A|>C(n,\theta)>0$.
Hence, the case when $n\geq 3$ follows from \eqref{int-d-r} \eqref{int-delta}, \eqref{int-n-3}, \eqref{int-lower-3}, and \eqref{def-Theta}. Lemma \ref{lem-Theta} is proved.
\end{proof}

For any $r>0$, we denote 
\begin{equation*}%\label{def-sigma-r}
\Sigma_r^\varepsilon := \left\{ x \in \mathbb{R}^n : d(x')=r, ~-\frac{\varepsilon}{2} + h_2(x') < x_n < h_2(x') \right\}
\end{equation*}
and
\begin{equation}\label{def-gamma=dist}
\widetilde\Gamma_{-,r}^\varepsilon:=\left\{x_n=-\frac{\varepsilon}{2}+h_2(x'),~d(x')<r\right\},
\end{equation}
where $d(x')$ is defined in \eqref{def-dist}.

\begin{prop}\label{prop-converg-F}
Let $n\geq 2$, $h_1$ and $h_2$ be two $C^2$ functions satisfying \eqref{assump-h}, $p>1$, $\varepsilon\in(0,1)$, and $u_\varepsilon\in W^{1,p}(\cD)$ be a weak solution of \eqref{p-laplace}. Then there exist constants $C_1,C_2>0$ depending on $n,p,c_1$, $c_2$, and $\|\varphi\|_{L^\infty(\partial\cD)}$,  such that
\begin{equation*}
\left|\lim_{\varepsilon\rightarrow0}\int_{\widetilde\Gamma_{-,r}^\varepsilon}|Du_\varepsilon|^{p-2}Du_\varepsilon\cdot\nu-\mathcal{F}\right|\leq C_1e^{-\frac{C_2}{r}},\quad r\in(0,1),
\end{equation*}
where $\mathcal{F}$ is given in \eqref{def-F}.
\end{prop}

\begin{proof}
For small $r\in (0, 1)$ and small $\varepsilon>0$, we take a smooth surface $\eta$ such that $\cD_1^\varepsilon$ is surrounded by $\widetilde\Gamma_{-,r}^0 \cup \eta$. 
Using the definition of $\mathcal{F}$ in \eqref{def-F} and  by integration by parts, we have
\[
-\int_{\widetilde\Gamma_{-,r}^0} |Du_0|^{p-2} Du_0 \cdot\nu + \int_\eta |Du_0|^{p-2} Du_0 \cdot \nu= \mathcal{F}.
\]
Similarly, from the third condition in \eqref{p-laplace}, we obtain 
\[
-\int_{\widetilde\Gamma_{-,r}^\varepsilon} |Du_\varepsilon|^{p-2} Du_\varepsilon \cdot \nu + \int_{\Sigma_r^\varepsilon} |Du_\varepsilon|^{p-2} Du_\varepsilon \cdot \nu + \int_\eta |Du_\varepsilon|^{p-2} Du_\varepsilon \cdot \nu = 0.
\]
By using \eqref{strong-converg}, we have 
\begin{equation*}
\lim_{\varepsilon\rightarrow0}\int_\eta |Du_\varepsilon|^{p-2} Du_\varepsilon \cdot \nu=\int_\eta |Du_0|^{p-2} Du_0 \cdot \nu.
\end{equation*}
By using \eqref{asym-Du} and the second condition in \eqref{varep=0}, we have 
\[
|Du_0(x)| \leq C_1 e^{-\frac{C_2}{r}} \quad\text{ in } ~\widetilde\Gamma_{-,r}^0.
\]
This implies that
\begin{align*}
\left|\int_{\widetilde\Gamma_{-,r}^0} |Du_0|^{p-2} Du_0 \cdot \nu \right| \leq C_1 e^{-\frac{C_2}{r}}.
\end{align*}
In view of \eqref{asym-Du}, and using \eqref{sigma'}, we derive 
\begin{align*}
\left|\int_{\Sigma_r^\varepsilon} |Du_\varepsilon|^{p-2} Du_\varepsilon \cdot \nu \right|&\leq C\int_{d(x')=r}\int_{-\frac{\varepsilon}{2} + h_2(x')}^{h_2(x')}\left(\frac{1}{\varepsilon+d(x')^2}\right)^{p-1}\ dx_ndx'\\
&\leq C\varepsilon\int_{d(x')=r}\left(\frac{1}{\varepsilon+d(x')^2}\right)^{p-1}\ dx'\\
&\leq \frac{C\varepsilon r^{n-2}}{(\varepsilon+r^{2})^{p-1}}\rightarrow0\quad\mbox{as}\quad \varepsilon\rightarrow0_+.
\end{align*}
Thus, 
\begin{equation*}
\left|\lim_{\varepsilon\rightarrow0_+}\int_{\widetilde\Gamma_{-,r}^\varepsilon}|Du_\varepsilon|^{p-2}Du_\varepsilon\cdot\nu-\mathcal{F}\right|\leq C_1e^{-\frac{C_2}{r}}.
\end{equation*}
The proof is finished.
\end{proof}

We are in a position to establish the estimate of $U_1^\varepsilon-U_2^\varepsilon$.

\begin{theorem}\label{thm-U12}
Let $p>1$, $U_1^\varepsilon$ and $U_2^\varepsilon$ be the constants in \eqref{p-laplace}. Then we have
\begin{equation}\label{lim-est-u1u2}
\lim_{\varepsilon\rightarrow0}\frac{U_1^{\varepsilon}-U_2^{\varepsilon}}{\Theta(\varepsilon;p)}=\text{sgn}(\mathcal{F})|\mathcal{F}|^{\frac{1}{p-1}},
\end{equation}
where $\mathcal{F}$ and $\Theta(\varepsilon;p)$ are defined in \eqref{def-F} and \eqref{def-Theta}, respectively.
\end{theorem}

\begin{proof}
It follows from \eqref{asym-Du} that
\begin{equation}\label{Du-nu}
\left|Du_\varepsilon\cdot\nu(y)-\frac{U_1^\varepsilon-U_2^\varepsilon}{\delta(y')}\right|\leq C_1\left(\frac{|U_1^\varepsilon-U_2^\varepsilon|}{\delta(y')^{1-\beta/2}}+e^{-\frac{C_2}{\sqrt\varepsilon+d(y')}}\|\varphi\|_{L^\infty(\partial\cD)}\right),\quad y\in \Gamma_{-,1/4}^\varepsilon.
\end{equation}
Without loss of generality, we shall let $\{\varepsilon_j\}_{j\in\mathbb N}$ be a decreasing sequence such that $\varepsilon_j\rightarrow0$ as $j\rightarrow\infty$ and $U_1^{\varepsilon_j}\geq U_2^{\varepsilon_j}$ for all $j\in\mathbb N$. Then by using \eqref{Du-nu}, we obtain, for any $\kappa\in(0,1)$, $r\in(0,1/4)$, and $y\in \Gamma_{-,1/4}^\varepsilon$,
\begin{align*}
&(1-\kappa)\left(\frac{U_1^{\varepsilon_j}-U_2^{\varepsilon_j}}{\delta_j(y')}\right)^{p-1}(1-C\delta_j(y')^{\beta/2})-C(\kappa)e^{-\frac{C_2}{\sqrt{\varepsilon_j}+d(y')}}\leq |Du_{\varepsilon_j}|^{p-2}Du_{\varepsilon_j}\cdot\nu(y)\\
&\leq (1+\kappa)\left(\frac{U_1^{\varepsilon_j}-U_2^{\varepsilon_j}}{\delta_j(y')}\right)^{p-1}(1+C\delta_j(y')^{\beta/2})+C(\kappa)e^{-\frac{C_2}{\sqrt{\varepsilon_j}+d(y')}},
\end{align*}
where $\delta_j(y'):=\varepsilon_j+h_1(y')-h_2(y')$, $C$, $C_2>0$ depend on $n,p,c_1,c_2$, $\beta$, $\|\varphi\|_{L^\infty(\partial\cD)}$, and $\mbox{dist}(\cD_1^{\varepsilon_j}\cup \cD_2^{\varepsilon_j},\partial \cD)$, and $C(\kappa)$ additionally depends on $\kappa$. For any $r\in(0,1/4)$, by using \eqref{assump-h}, we have 
\begin{align}\label{int-U1U2}
&(1-\kappa)\int_{d(y')<r}\left(\frac{U_1^{\varepsilon_j}-U_2^{\varepsilon_j}}{\delta_j(y')}\right)^{p-1}(1-C\delta_j(y')^{\beta/2})\sqrt{1+|D_{y'}h_2(y')|^2}\ dy'\nonumber\\
&\quad-C(\kappa)\int_{d(y')<r}e^{-\frac{C_2}{\sqrt{\varepsilon_j}+d(y')}}\ dy'\nonumber\\
&\leq\int_{\widetilde\Gamma_{-,r}^{\varepsilon_j}}|Du_{\varepsilon_j}|^{p-2}Du_{\varepsilon_j}\cdot\nu\ dS\nonumber\\
&\leq(1+\kappa)\int_{d(y')<r}\left(\frac{U_1^{\varepsilon_j}-U_2^{\varepsilon_j}}{\delta_j(y')}\right)^{p-1}(1+C\delta_j(y')^{\beta/2})\sqrt{1+|D_{y'}h_2(y')|^2}\ dy'\nonumber\\
&\quad+C(\kappa)\int_{d(y')<r}e^{-\frac{C_2}{\sqrt{\varepsilon_j}+d(y')}}\ dy'.
\end{align}
Taking the limit as $j\rightarrow\infty$ in \eqref{int-U1U2} first, then taking $r\rightarrow0$ and $\kappa\rightarrow0$, combining Lemma \ref{lem-Theta} and  Proposition \ref{prop-converg-F},  we derive 
\begin{equation}\label{u1-u2-est}
\lim_{j\rightarrow\infty}\left(\frac{U_1^{\varepsilon_j}-U_2^{\varepsilon_j}}{\Theta(\varepsilon_j;p)}\right)^{p-1}=\mathcal{F}.
\end{equation}
Similarly, if $\{\varepsilon_j\}_{j\in\mathbb N}$ is a decreasing sequence such that $\varepsilon_j\rightarrow0$ as $j\rightarrow\infty$ and $U_1^{\varepsilon_j}\leq U_2^{\varepsilon_j}$ for all $j\in\mathbb N$, then we have 
\begin{equation}\label{u2-u1-est}
\lim_{j\rightarrow\infty}\left(\frac{U_2^{\varepsilon_j}-U_1^{\varepsilon_j}}{\Theta(\varepsilon_j;p)}\right)^{p-1}=-\mathcal{F}.
\end{equation}
Therefore, if $\mathcal{F}>0$, for any decreasing sequence $\varepsilon_j\rightarrow0_+$ as $j\rightarrow\infty$, there exists $j_0\in\mathbb N$ such that for any $j\geq j_0$, we have $U_1^{\varepsilon_j}\geq U_2^{\varepsilon_j}$. Thus,  it follows from \eqref{u1-u2-est} that \eqref{lim-est-u1u2} holds. If $\mathcal{F}<0$, we obtain \eqref{lim-est-u1u2} from \eqref{u2-u1-est}. If $\mathcal{F}=0$, we let $\{\varepsilon_j\}_{j\in\mathbb N}$ be any sequence such that $\varepsilon_j\rightarrow0$ as $j\rightarrow\infty$. Then there exists a decreasing subsequence $\{\varepsilon_{j_k}\}$ such that either $U_1^{\varepsilon_{j_k}}\geq U_2^{\varepsilon_{j_k}}$ holds for all $k\in\mathbb N$ or $U_1^{\varepsilon_{j_k}}\leq U_2^{\varepsilon_{j_k}}$ holds for all $k\in\mathbb N$. This together with \eqref{u1-u2-est} or \eqref{u2-u1-est} with $\mathcal{F}=0$ implies our desired result.
The proof of Theorem \ref{thm-U12} is finished.
\end{proof}

Now, we are ready to prove Theorem \ref{thm-flat}.
\begin{proof}[Proof of Theorem \ref{thm-flat}]
Denote 
\begin{equation*}
f_1(\varepsilon):=\frac{U_1^\varepsilon-U_2^\varepsilon}{\Theta(\varepsilon;p)}-\text{sgn}(\mathcal{F})|\mathcal{F}|^{\frac{1}{p-1}}.
\end{equation*}
It follows from Theorem \ref{thm-U12} that
\begin{align*}
\lim_{\varepsilon\rightarrow0}f_1(\varepsilon)=0.
\end{align*}
Combining \eqref{def-Theta} and \eqref{asym-Du}, we have 
\begin{align*}
Du_{\varepsilon}=\left(0',\frac{\Theta(\varepsilon;p)}{\delta(x')}\big(\text{sgn}(\mathcal{F})|\mathcal{F}|^{\frac{1}{p-1}}+f_1(\varepsilon)\big)\right)
+{\bf f}_1(x,\varepsilon),
\end{align*}
where 
$$|{\bf f}_1(x,\varepsilon)|\leq C_1\left(\frac{\Theta(\varepsilon;p)}{\delta(x')^{1-\beta/2}}\big(|\mathcal{F}|^{\frac{1}{p-1}}+|f_1(\varepsilon)|\big)+e^{-\frac{C_2}{\sqrt\varepsilon+d(x')}}\|\varphi\|_{L^\infty(\partial\cD)}\right),$$
$\beta\in(0,1)$, $C_1>0$ and $C_2>0$ depend on $n,p,c_1$, and $c_2$.
We complete the proof of Theorem \ref{thm-flat}. 
\end{proof}

\section{Proof of Theorem \ref{thm-gamma}}\label{sec-thm2}
This section is devoted to the proof of Theorem \ref{thm-gamma}. 

\begin{lemma}\label{lem-Theta-gamma}
Let $p\geq \frac{n+\gamma}{1+\gamma}$, $n\geq 2$, and $r\in(0,1)$. We have
\begin{align*}
\frac{1}{K}c_0^{-\frac{n-1}{1+\gamma}}&\leq\liminf_{\varepsilon\rightarrow0}\int_{|x'|<r}\left(\frac{\Theta(\varepsilon;p,\gamma)}{\delta(x')}\right)^{p-1}\ dx'\\
&\leq\limsup_{\varepsilon\rightarrow0}\int_{|x'|<r}\left(\frac{\Theta(\varepsilon;p,\gamma)}{\delta(x')}\right)^{p-1}\ dx'\leq\frac{1}{K}c_0^{\frac{n-1}{1+\gamma}},
\end{align*}
where $K$, $c_0$, $\delta(x')$, and $\Theta(\varepsilon;p,\gamma)$  are defined in \eqref{def-K0}, \eqref{est-a}, \eqref{def-delta}, and \eqref{def-Theta-gamma}, respectively.
\end{lemma}

\begin{proof}
By using \eqref{delta-asym} and a suitable change of variables, we have 
\begin{align}\label{int-delta-gamma}
\int_{|x'|<r}\left(\frac{1}{\delta(x')}\right)^{p-1}\ dx'&=\int_{|x'|<r}\frac{1}{(\varepsilon+a(x')|x'|^{1+\gamma})^{p-1}}\ dx'\nonumber\\
&=\varepsilon^{\frac{n+\gamma}{1+\gamma}-p}\int_{\mathbb{S}^{n-2}}\int_{0}^{r\varepsilon^{-\frac{1}{1+\gamma}}}\frac{s^{n-2}}{(1+a(s\varepsilon^{\frac{1}{1+\gamma}}\theta)s^{1+\gamma})^{p-1}}\ dsd\sigma(\theta),
\end{align}
where $|\mathbb{S}^{n-2}|=\frac{2\pi^{\frac{n-1}{2}}}{\Gamma(\frac{n-1}{2})}$. By using \eqref{est-a}, we have 
\begin{align}\label{int-K}
\omega_{n-2}c_0^{-\frac{n-1}{1+\gamma}}\int_{0}^{r(c_0^{-1}\varepsilon)^{-\frac{1}{1+\gamma}}}\frac{s^{n-2}}{(1+s^{1+\gamma})^{p-1}}\ ds&\leq \int_{\mathbb{S}^{n-2}}\int_{0}^{r\varepsilon^{-\frac{1}{1+\gamma}}}\frac{s^{n-2}}{(1+a(s\varepsilon^{\frac{1}{1+\gamma}}\theta)s^{1+\gamma})^{p-1}}\ dsd\sigma(\theta)\nonumber\\
&\leq \omega_{n-2}c_0^{\frac{n-1}{1+\gamma}}\int_{0}^{r(c_0\varepsilon)^{-\frac{1}{1+\gamma}}}\frac{s^{n-2}}{(1+s^{1+\gamma})^{p-1}}\ ds,
\end{align}
where $\omega_{n-2}=|\mathbb{S}^{n-2}|=\frac{2\pi^{\frac{n-1}{2}}}{\Gamma(\frac{n-1}{2})}$, and $c_0>1$ is a constant given in \eqref{est-a}. 

When $p>\frac{n+\gamma}{1+\gamma}$, $\Theta(\varepsilon;p,\gamma)=
\varepsilon^{1-\frac{n-1}{(1+\gamma)(p-1)}}$. Then we have 
\begin{align*}%\label{est-case1}
&\omega_{n-2}\lim_{\varepsilon\rightarrow0}\int_{0}^{(c_0r)\varepsilon^{-\frac{1}{1+\gamma}}}\frac{s^{n-2}}{(1+s^{1+\gamma})^{p-1}}\ ds=\omega_{n-2}\int_0^{\infty}\frac{s^{n-2}}{(1+s^{1+\gamma})^{p-1}}\ ds\nonumber\\
&=\frac{\omega_{n-2}}{1+\gamma}\int_0^{\infty}\frac{s^{\frac{n-2-\gamma}{1+\gamma}}}{(1+s)^{p-1}}\ ds=\frac{\omega_{n-2}}{1+\gamma}B\left(\frac{n-1}{1+\gamma},p-\frac{n+\gamma}{1+\gamma}\right)\nonumber\\
&=\frac{2\pi^{\frac{n-1}{2}}}{1+\gamma}\cdot\frac{\Gamma(\frac{n-1}{1+\gamma})\Gamma(p-\frac{n+\gamma}{1+\gamma})}{\Gamma(\frac{n-1}{2})\Gamma(p-1)},
\end{align*}
where $B$ is the beta function. Thus, we obtain from \eqref{int-delta-gamma} and \eqref{int-K} that 
\begin{align*}
&\frac{2\pi^{\frac{n-1}{2}}}{c_0^{\frac{n-1}{1+\gamma}}(1+\gamma)}\cdot\frac{\Gamma(\frac{n-1}{1+\gamma})\Gamma(p-\frac{n+\gamma}{1+\gamma})}{\Gamma(\frac{n-1}{2})\Gamma(p-1)}
\leq\liminf_{\varepsilon\rightarrow0}\int_{|x'|<r}\left(\frac{\Theta(\varepsilon;p,\gamma)}{\delta(x')}\right)^{p-1}\ dx'\\
&\leq\limsup_{\varepsilon\rightarrow0}\int_{|x'|<r}\left(\frac{\Theta(\varepsilon;p,\gamma)}{\delta(x')}\right)^{p-1}\ dx'\leq\frac{2c_0^{\frac{n-1}{1+\gamma}}\pi^{\frac{n-1}{2}}}{1+\gamma}\cdot\frac{\Gamma(\frac{n-1}{1+\gamma})\Gamma(p-\frac{n+\gamma}{1+\gamma})}{\Gamma(\frac{n-1}{2})\Gamma(p-1)}.
\end{align*}

When $p=\frac{n+\gamma}{1+\gamma}$, $\Theta(\varepsilon;p,\gamma)=
|\ln\varepsilon|^{-\frac{1}{p-1}}$. Then we have 
\begin{align*}
&|\ln\varepsilon|^{-1}\int_0^{r(c_0\varepsilon)^{-\frac{1}{1+\gamma}}}\frac{s^{n-2}}{(1+s^{1+\gamma})^{\frac{n-1}{1+\gamma}}}\ ds
=|\ln\varepsilon|^{-1}\int_0^{r(c_0\varepsilon)^{-\frac{1}{1+\gamma}}}\frac{s^{\gamma}}{1+s^{1+\gamma}}\ ds\\
&\qquad+|\ln\varepsilon|^{-1}\int_0^{r(c_0\varepsilon)^{-\frac{1}{1+\gamma}}}\left(\frac{s^{n-2}}{(1+s^{1+\gamma})^{\frac{n-1}{1+\gamma}}}-\frac{s^{\gamma}}{1+s^{1+\gamma}}\right)\ ds=:\mbox{I}+\mbox{II}.
\end{align*}
Next, we estimate $\mbox{I}$ and $\mbox{II}$, respectively. It is clear that 
\begin{align*}
\lim_{\varepsilon\rightarrow0}\mbox{I}=\lim_{\varepsilon\rightarrow0}\frac{|\ln\varepsilon|^{-1}}{1+\gamma}\ln(1+r^{1+\gamma}(c_0\varepsilon)^{-1})=\frac{1}{1+\gamma}.
\end{align*}
For small $\varepsilon>0$, we divide $\mbox{II}$ as follows:
\begin{align}\label{est-II}
\mbox{II}&=|\ln\varepsilon|^{-1}\Bigg(\int_0^1\left(\frac{s^{n-2}}{(1+s^{1+\gamma})^{\frac{n-1}{1+\gamma}}}-\frac{s^{\gamma}}{1+s^{1+\gamma}}\right)\ ds\nonumber\\
&\quad+\int_1^{r(c_0\varepsilon)^{-\frac{1}{1+\gamma}}}\left(\frac{s^{n-2}}{(1+s^{1+\gamma})^{\frac{n-1}{1+\gamma}}}-\frac{s^{\gamma}}{1+s^{1+\gamma}}\right)\ ds\Bigg)=:\mbox{II}_1+\mbox{II}_2.
\end{align}
Note that 
\begin{align}\label{est-II1}
|\mbox{II}_1|\leq C|\ln\varepsilon|^{-1}
\end{align}
and 
\begin{equation*}
\mbox{II}_2=|\ln\varepsilon|^{-1}\int_1^{r(c_0\varepsilon)^{-\frac{1}{1+\gamma}}}\frac{s^\gamma\Big(s^{n-2-\gamma}-(1+s^{1+\gamma})^{\frac{n-2-\gamma}{1+\gamma}}\Big)}{(1+s^{1+\gamma})^{\frac{n-1}{1+\gamma}}}\ ds.
\end{equation*}
By using the mean value theorem, there exists a $\xi\in(s^{1+\gamma},1+s^{1+\gamma})\subset(s^{1+\gamma},2s^{1+\gamma})$ (when $s\geq1$), such that 
\begin{equation*}
s^{n-2-\gamma}-(1+s^{1+\gamma})^{\frac{n-2-\gamma}{1+\gamma}}=-\frac{n-2-\gamma}{1+\gamma}\xi^{\frac{n-3-2\gamma}{1+\gamma}}.
\end{equation*}
Thus, we have 
\begin{align}\label{est-II2}
 |\mbox{II}_2|\leq C|\ln\varepsilon|^{-1}\int_1^{\infty}\frac{s^{n-3-\gamma}}{(1+s^{1+\gamma})^{\frac{n-1}{1+\gamma}}}\ ds\leq C|\ln\varepsilon|^{-1}.
\end{align}
Substituting \eqref{est-II1} and \eqref{est-II2} into \eqref{est-II}, we derive 
\begin{equation*}
\lim_{\varepsilon\rightarrow0}\mbox{II}=0.
\end{equation*}
Therefore, we have 
\begin{equation*}%\label{est-case2}
\omega_{n-2}\lim_{\varepsilon\rightarrow0}|\ln\varepsilon|^{-1}\int_0^{r(c_0\varepsilon)^{-\frac{1}{1+\gamma}}}\frac{s^{n-2}}{(1+s^{1+\gamma})^{\frac{n-1}{1+\gamma}}}\ ds=\frac{\omega_{n-2}}{1+\gamma}=\frac{2\pi^{\frac{n-1}{2}}}{(1+\gamma)\Gamma(\frac{n-1}{2})}
\end{equation*}
and 
\begin{align*}
\frac{2\pi^{\frac{n-1}{2}}}{c_0^{\frac{n-1}{1+\gamma}}(1+\gamma)\Gamma(\frac{n-1}{2})}
&\leq\liminf_{\varepsilon\rightarrow0}\int_{|x'|<r}\left(\frac{\Theta(\varepsilon;p,\gamma)}{\delta(x')}\right)^{p-1}\ dx'\\
&\leq\limsup_{\varepsilon\rightarrow0}\int_{|x'|<r}\left(\frac{\Theta(\varepsilon;p,\gamma)}{\delta(x')}\right)^{p-1}\ dx'\leq\frac{2c_0^{\frac{n-1}{1+\gamma}}\pi^{\frac{n-1}{2}}}{(1+\gamma)\Gamma(\frac{n-1}{2})}.
\end{align*} 
The proof is completed.
\end{proof}

Following the proof of Proposition \ref{prop-converg-F} with slight modifications, we have the result as follows.

\begin{prop}\label{prop-gamma}
Let $p\geq\frac{n+\gamma}{1+\gamma}$,  $n\geq 2$, $h_1$ and $h_2$ be two $C^{1,\gamma}$ functions satisfying \eqref{assump-h-gamma}, $\varepsilon\in(0,1)$, and let $u_\varepsilon\in W^{1,p}(\cD)$ be a weak solution of \eqref{p-laplace}. Then for any $r>0$, we have
\begin{equation*}
\left|\lim_{\varepsilon\rightarrow0_+}\int_{\Gamma_{-,r}^\varepsilon}|Du_\varepsilon|^{p-2}Du_\varepsilon\cdot\nu-\mathcal{F}\right|\leq C_1e^{-\frac{C_2}{r}},
\end{equation*}
where $\mathcal{F}$ is given in \eqref{def-F}, $C_1>0$ depends on $n,p,\gamma,c_3$, $c_4$, and $\|\varphi\|_{L^\infty(\partial\cD)}$, and $C_2>0$ depends on $n,p,c_3$, $c_4$, and $\|\varphi\|_{L^\infty(\partial\cD)}$.
\end{prop}

\begin{theorem}\label{thm-U12-gamma}
Let $p\geq\frac{n+\gamma}{1+\gamma}$, and $U_1^\varepsilon$ and $U_2^\varepsilon$ be the constants in \eqref{p-laplace}. Then we have
\begin{align*}
(c_0^{-\frac{n-1}{1+\gamma}}K\mathcal{F})^{\frac{1}{p-1}}&\leq\liminf_{\varepsilon\rightarrow0}\frac{U_1^{\varepsilon}-U_2^{\varepsilon}}{\Theta(\varepsilon;p,\gamma)}\\
&\leq\limsup_{\varepsilon\rightarrow0}\frac{U_1^{\varepsilon}-U_2^{\varepsilon}}{\Theta(\varepsilon;p,\gamma)}\leq(c_0^{\frac{n-1}{1+\gamma}}K\mathcal{F})^{\frac{1}{p-1}},\quad\mbox{if}\quad\mathcal{F}\geq 0,
\end{align*}
and 
\begin{align*}
-(-c_0^{\frac{n-1}{1+\gamma}}K\mathcal{F})^{\frac{1}{p-1}}&\leq\liminf_{\varepsilon\rightarrow0}\frac{U_1^{\varepsilon}-U_2^{\varepsilon}}{\Theta(\varepsilon;p,\gamma)}\\
&\leq\limsup_{\varepsilon\rightarrow0}\frac{U_1^{\varepsilon}-U_2^{\varepsilon}}{\Theta(\varepsilon;p,\gamma)}\leq-(-c_0^{-\frac{n-1}{1+\gamma}}K\mathcal{F})^{\frac{1}{p-1}},\quad\mbox{if}\quad\mathcal{F}<0,
\end{align*}
where $c_0$, $\Theta(\varepsilon;p,\gamma)$, $K$, and $\mathcal{F}$ are defined in \eqref{est-a}, \eqref{def-Theta-gamma}, \eqref{def-K0}, and \eqref{def-F} with $p\geq\frac{n+\gamma}{1+\gamma}$, respectively.
\end{theorem}

\begin{proof}
The desired result follows from the argument in the proof of Theorem \ref{thm-U12} together with Lemma \ref{lem-Theta-gamma} and Proposition \ref{prop-gamma}. Thus, we omit the details.
\end{proof}

Now we are in a position to prove Theorem \ref{thm-gamma}.
\begin{proof}[Proof of Theorem \ref{thm-gamma}]
When $p\geq\frac{n+\gamma}{1+\gamma}$, by using  Theorem \ref{thm-U12-gamma}, for sufficiently small $\varepsilon>0$, we have 
\begin{equation}\label{U1-U2-bdd}
\frac{1}{C_0}(K|\mathcal{F}|)^{\frac{1}{p-1}}\leq\frac{|U_1^\varepsilon-U_2^\varepsilon|}{\Theta(\varepsilon;p,\gamma)}\leq C_0(K|\mathcal{F}|)^{\frac{1}{p-1}},
\end{equation}
where $C_0>1$ is a constant depending on $n,p,\gamma,c_3$, and $c_4$. This together with \eqref{asym-Du-gamma} yields 
\begin{align*}
|Du_\varepsilon(x)|\leq C(K|\mathcal{F}|)^{\frac{1}{p-1}}\frac{\Theta(\varepsilon;p,\gamma)}{\delta(x')}(1+\delta(x')^{\beta/2})+C_1e^{-\frac{C_2}{\varepsilon^{\gamma/(1+\gamma)}+|x'|^\gamma}}\|\varphi\|_{L^\infty(\partial\cD)},
\end{align*}
where $C,C_1>0$ are constants depending on $n,p,\gamma,c_3$, and $c_4$, and $C_2>0$ depends on $n,p,c_3$, and $c_4$.
For any $x=(0',x_n)\in\Omega_R^\varepsilon$, by using \eqref{U1-U2-bdd} and \eqref{asym-Du-gamma}, we have 
\begin{align*}
|Du_\varepsilon(0',x_n)|&\geq\frac{1}{C_0}(K|\mathcal{F}|)^{\frac{1}{p-1}}\frac{\Theta(\varepsilon;p,\gamma)}{\varepsilon}-C(K|\mathcal{F}|)^{\frac{1}{p-1}}\frac{\Theta(\varepsilon;p,\gamma)}{\varepsilon^{1-\beta/2}}-C_1e^{-\frac{C_2}{\varepsilon^{\gamma/(1+\gamma)}}}\|\varphi\|_{L^\infty(\partial\cD)}\\
&=(K|\mathcal{F}|)^{\frac{1}{p-1}}\frac{\Theta(\varepsilon;p,\gamma)}{\varepsilon}\big(\frac{1}{C_0}-C\varepsilon^{\beta/2}\big)-C_1e^{-\frac{C_2}{\varepsilon^{\gamma/(1+\gamma)}}}\|\varphi\|_{L^\infty(\partial\cD)}.
\end{align*}
By choosing $\varepsilon>0$ small enough such that $\frac{1}{C_0}-C\varepsilon^{\beta/2}\geq\frac{1}{2C_0}$, we obtain
\begin{equation*}
|Du_\varepsilon(0',x_n)|\geq\frac{1}{2C_0}(K|\mathcal{F}|)^{\frac{1}{p-1}}\frac{\Theta(\varepsilon;p,\gamma)}{\varepsilon}-C_1e^{-\frac{C_2}{\varepsilon^{\gamma/(1+\gamma)}}}\|\varphi\|_{L^\infty(\partial\cD)}.
\end{equation*}

When  $1<p<\frac{n+\gamma}{1+\gamma}$, we have from \eqref{asym-Du-gamma} that
\begin{equation*}
|Du_\varepsilon(x)|\leq C(\varepsilon+|x'|^{1+\gamma})^{-1}.
\end{equation*}
Thus, we have 
\begin{equation*}
\left|\int_{\Gamma_{-,s}^0}|Du_\varepsilon|^{p-2}Du_\varepsilon\cdot\nu\right|\leq C\int_{|x'|<s}\frac{1}{|x'|^{(1+\gamma)(p-1)}}\leq Cs^{n+\gamma-p(1+\gamma)},\quad s\in(0,1),
\end{equation*}
which replaces the term $Cs^{n+1-2p}$ in the proof of \cite[Theorem 2.5]{dyz2024} when $p<(n+1)/2$. By subsequently following the argument in the same proof, we obtain
$$\lim_{\varepsilon\rightarrow0}f_{2}(\varepsilon)=0,$$
where 
$$f_{2}(\varepsilon):=U_1^\varepsilon-U_2^\varepsilon-(U_1-U_2).$$
This in combination with \eqref{asym-Du-gamma} finishes the proof of Theorem \ref{thm-gamma}.
\end{proof}

We end this section with two special cases as follows.

\begin{remark}\label{rmk-U1-U2-x'}
When $a=a(x'/|x'|)$ in \eqref{delta-asym}, we denote 
\begin{align*}%\label{def-K011}
K_0:=\begin{cases}
\big(\int_{\mathbb{S}^{n-2}}\int_{0}^{\infty}\frac{s^{n-2}}{(1+a(\theta)s^{1+\gamma})^{p-1}}\ dsd\sigma(\theta)\big)^{-1},&\quad p>\frac{n+\gamma}{1+\gamma},\\
\big(\lim\limits_{\varepsilon\rightarrow0}|\ln\varepsilon|^{-1}\int_{\mathbb{S}^{n-2}}\int_0^{r\varepsilon^{-\frac{1}{1+\gamma}}}\frac{s^{n-2}}{(1+a(\theta)s^{1+\gamma})^{\frac{n-1}{1+\gamma}}}\ dsd\sigma(\theta)\big)^{-1},&\quad p=\frac{n+\gamma}{1+\gamma}.
\end{cases}
\end{align*}
Then it follows from \eqref{int-delta-gamma} that
\begin{equation*}
\lim_{\varepsilon\rightarrow0}\int_{|x'|<r}\left(\frac{\Theta(\varepsilon;p,\gamma)}{\delta(x')}\right)^{p-1}\ dx'=\frac{1}{K_0}.
\end{equation*}
Thus, as in the proof of Theorem \ref{thm-U12-gamma}, we obtain 
\begin{equation*}
\lim_{\varepsilon\rightarrow0}\frac{U_1^{\varepsilon}-U_2^{\varepsilon}}{\Theta(\varepsilon;p,\gamma)}=\text{sgn}(\mathcal{F})(K_0|\mathcal{F}|)^{\frac{1}{p-1}}.
\end{equation*}
Consequently, for small $\varepsilon\in(0,100^{-2})$, $x\in\Omega_{R}^\varepsilon$, and $p\geq \frac{n+\gamma}{1+\gamma}$, we have 
\begin{align*}
Du_\varepsilon(x)=\big(0',\delta(x')^{-1}\Theta(\varepsilon;p,\gamma)(\text{sgn}(\mathcal{F})(K_0|\mathcal{F}|)^{\frac{1}{p-1}}+f_0(\varepsilon))\big)+{\bf f}_0(x,\varepsilon),
\end{align*}
where $f_0: \mathbb R\rightarrow\mathbb R$ is a function of $\varepsilon$, and ${\bf f}_0: \mathbb R^n\times\mathbb R\rightarrow\mathbb R^n$ is a function of $x$ and $\varepsilon$ satisfying 
\begin{equation*}
\lim_{\varepsilon\rightarrow0}f_0(\varepsilon)=0
\end{equation*}
and 
\begin{align*}
|{\bf f}_0(x,\varepsilon)|\leq C\left(\delta(x')^{\beta/2-1}\Theta(\varepsilon;p,\gamma)((K_0|\mathcal{F}|)^{\frac{1}{p-1}}+| f_0(\varepsilon)|)\right)+C_1e^{-\frac{C_2}{\varepsilon^{\gamma/(1+\gamma)}+|x'|^\gamma}}\|\varphi\|_{L^\infty(\partial\cD)}.
\end{align*}
\end{remark}

\begin{remark}\label{rmk-U1U2-11}
In particular, if $a=a_0$ for some constant $a_0$, then from the proof of Lemma \ref{lem-Theta-gamma}, we have
\begin{equation*}
\lim_{\varepsilon\rightarrow0}\int_{|x'|<r}\left(\frac{\Theta(\varepsilon;p,\gamma)}{\delta(x')}\right)^{p-1}\ dx'=\frac{1}{K}a_0^{\frac{1-n}{1+\gamma}}
\end{equation*}
and
\begin{equation*}
\lim_{\varepsilon\rightarrow0}\frac{U_1^\varepsilon-U_2^\varepsilon}{\Theta(\varepsilon;p,\gamma)}=\text{sgn}(\mathcal{F})(Ka_0^{\frac{n-1}{1+\gamma}}|\mathcal{F}|)^{\frac{1}{p-1}}.
\end{equation*}
Thus, for small $\varepsilon\in(0,100^{-2})$, $x\in\Omega_{R}^\varepsilon$, and $p\geq \frac{n+\gamma}{1+\gamma}$, we have 
\begin{align*}
Du_\varepsilon(x)=\big(0',\delta(x')^{-1}\Theta(\varepsilon;p,\gamma)(\text{sgn}(\mathcal{F})(Ka_0^{\frac{n-1}{1+\gamma}}|\mathcal{F}|)^{\frac{1}{p-1}}+f(\varepsilon))\big)+{\bf f}(x,\varepsilon),
\end{align*}
where  $f: \mathbb R\rightarrow\mathbb R$ is a function of $\varepsilon$ and ${\bf f}: \mathbb R^n\times\mathbb R\rightarrow\mathbb R^n$ is a function of $x$ and $\varepsilon$ satisfying 
\begin{equation*}
\lim_{\varepsilon\rightarrow0}f(\varepsilon)=0
\end{equation*}
and 
\begin{align*}
|{\bf f}(x,\varepsilon)|\leq C\left(\delta(x')^{\beta/2-1}\Theta(\varepsilon;p,\gamma)((K|a_0|^{\frac{n-1}{1+\gamma}}|\mathcal{F}|)^{\frac{1}{p-1}}+|f(\varepsilon)|)\right)+C_1e^{-\frac{C_2}{\varepsilon^{\gamma/(1+\gamma)}+|x'|^\gamma}}\|\varphi\|_{L^\infty(\partial\cD)}.
\end{align*}
\end{remark}

\appendix

\section{Proof of Lemma \ref{lem-Dv}}\label{Append}
In this section, we give the proof of Lemma \ref{lem-Dv}.

\begin{proof}[Proof of Lemma \ref{lem-Dv}]
{\bf Case 1}: $h_1$ and $h_2$ are two $C^2$ functions satisfying \eqref{assump-h}.
As in the proof of \cite[Lemma 2.1]{dyz2024}, we may assume $\varepsilon\in(0,1/256]$ and $|x'|<1/16$ since otherwise the results in Lemma \ref{lem-Dv} follow from classical estimates for the $p$-Laplace equation (see, for instance, \cite{l1988}). Denote 
\begin{equation*}
\mathcal{C}_r^{\varepsilon}:=\left\{(x',x_n)\in\Omega^\varepsilon: -\frac{\varepsilon}{2}+h_2(x')<x_n<\frac{\varepsilon}{2}+h_1(x'), ~d(x')<r\right\},\quad r\in(0,1),
\end{equation*}
where $d(x')$ is defined in \eqref{def-dist}.
For any $0<t<s<1-R_0$ with $0<R_0<1/3$ given in \eqref{sigma'}, let $\eta(x')$ be a cutoff function such that $\eta=1$ in $\mathcal{C}_t^\varepsilon$ and $\eta=0$ in $\mathcal{C}_1^\varepsilon\setminus\mathcal{C}_s^\varepsilon$, and $\eta=\frac{s-d(x')}{s-t}$ in $\mathcal{C}_s^\varepsilon$. Note that $|D\eta|\le \frac{1}{s-t}$. Multiplying $v\eta^p$ on both sides of \eqref{eq-v-p} and integrating by parts, we have
\begin{equation*}
\int_{\mathcal{C}_1^\varepsilon}|Dv|^p\eta^p+p|Dv|^{p-2}Dv\cdot D\eta v\eta^{p-1}=0.
\end{equation*}
Then by using Young's inequality, we obtain
\begin{align*}
\int_{\mathcal{C}_1^\varepsilon}|Dv|^p\eta^p\leq p\int_{\mathcal{C}_1^\varepsilon}|Dv|^{p-1}|D\eta| |v|\eta^{p-1}\leq \frac{1}{2}\int_{\mathcal{C}_1^\varepsilon}|Dv|^p\eta^p+C\int_{\mathcal{C}_1^\varepsilon}|D\eta|^p |v|^p,
\end{align*}
which  implies 
\begin{equation}\label{est-v-Dv}
\int_{\mathcal{C}_t^\varepsilon}|Dv|^p\leq\frac{C}{(s-t)^p}\int_{\mathcal{C}_s^\varepsilon\setminus\mathcal{C}_t^\varepsilon}|v|^p.
\end{equation}
For any $x\in\mathcal{C}_s^\varepsilon$, we have $|x'|<s+R_0<1$. Thus, by using $v=0$ on $\Gamma_-^{\varepsilon}$, the Poincar\'{e} inequality, and \eqref{est-v-Dv}, we have 
\begin{align}\label{ite-Dv-flat2zz}
\int_{\mathcal{C}_t^\varepsilon}|Dv|^p&\leq \frac{C}{(s-t)^p}\int_{\mathcal{C}_s^\varepsilon\setminus\mathcal{C}_t^\varepsilon}(\varepsilon+d(x')^2)^p|Dv|^p\leq C_0\left(\frac{\varepsilon+s^2}{s-t}\right)^p\int_{\mathcal{C}_s^\varepsilon\setminus\mathcal{C}_t^\varepsilon}|Dv|^p.
\end{align}
Take $t_0=r\in(\sqrt\varepsilon,1/2)$ and $t_j=\big(1-\frac{1}{2}jr\big)r$ with $j\in\mathbb N$ satisfying $j\leq 2r^{-1}$. Then we take $s=t_j$ and $t=t_{j+1}$ in \eqref{ite-Dv-flat2zz}, we obtain
\begin{equation*}
\int_{\mathcal{C}_{t_{j+1}}^\varepsilon}|Dv|^p\leq 4^pC_0\int_{\mathcal{C}_{t_j}^\varepsilon\setminus\mathcal{C}_{t_{j+1}}^\varepsilon}|Dv|^p.
\end{equation*}
Adding both sides by $4^pC_0\int_{\mathcal{C}_{t_{j+1}}^\varepsilon}|Dv|^p$ and dividing both sides by $1+4^pC_0$, we obtain
\begin{equation*}
\int_{\mathcal{C}_{t_{j+1}}^\varepsilon}|Dv|^p\leq \frac{4^pC_0}{1+4^pC_0}\int_{\mathcal{C}_{t_j}^\varepsilon}|Dv|^p.
\end{equation*}
We then choose $k=\lfloor\frac{1}{2r}\rfloor$ and iterate the above inequality $k$ times. By using \eqref{est-v-Dv} and $r<1/2$,  we have 
\begin{equation*}
\int_{\mathcal{C}_{3r/4}^\varepsilon}|Dv|^p\leq \left(\frac{4^pC_0}{1+4^pC_0}\right)^k\int_{\mathcal{C}_{r}^\varepsilon}|Dv|^p\leq C\mu_1^{\frac{1}{r}}\int_{\mathcal{C}_{1-R_0}^\varepsilon\setminus\mathcal{C}_{1/2}^\varepsilon}|v|^p,
\end{equation*}
where $\mu_1\in(0,1)$ and $C>0$ are constants depending on $n,p,c_1$, and $c_2$. By using \eqref{def-Omega} and $d(x')\leq |x'|$, we have  $\Omega_{1/2}^\varepsilon\subset\mathcal{C}_{1/2}^\varepsilon$ and $\mathcal{C}_{1-R_0}^\varepsilon\subset\Omega_{1}^\varepsilon$, where $\Omega_{r}^\varepsilon$ is defined in \eqref{def-Omega}. Thus, we obtain
\begin{equation}\label{est-Dv-p-v}
\int_{\mathcal{C}_{3r/4}^\varepsilon}|Dv|^p\leq C\mu_1^{\frac{1}{r}}\int_{\Omega_{1}^\varepsilon\setminus\Omega_{1/2}^\varepsilon}|v|^p.
\end{equation}
For any $x\in\overline{\Omega_{1/3}^\varepsilon}$ and small $\varepsilon\in(0,100^{-2})$, we take $\rho=\frac{1}{6}(\varepsilon+d(x')^2)$ and  $r=\frac{17}{12}(\sqrt\varepsilon+d(x'))<1/2$. Then for any $y\in\Omega_{\rho}^\varepsilon(x)$ given in \eqref{def-Omega}, we have 
\begin{align*}
d(y')\leq d(x')+\rho< 3r/4,
\end{align*}
and hence, $\Omega_{\rho}^\varepsilon(x)\subset \mathcal{C}_{3r/4}^\varepsilon$.
By applying classical estimates for the $p$-Laplace equation in $\Omega_{\rho}^\varepsilon(x)$ with $x\in\overline{\Omega_{1/3}^\varepsilon}$ (see, for instance, \cite{l1988}), and using \eqref{est-Dv-p-v}, we derive
\begin{align}\label{est-Dv-p}
|Dv(x)|&\leq \frac{C}{\rho^{n/p}}\left(\int_{\Omega_{\rho}^\varepsilon(x)}|Dv|^p\right)^{1/p}\nonumber\\
&\leq \frac{C}{r^{2n/p}}\left(\int_{\mathcal{C}_{3r/4}^\varepsilon}|Dv|^p\right)^{1/p}\nonumber\\
&\leq \frac{C}{r^{2n/p}}\mu_1^{\frac{1}{pr}}\|v\|_{L^p(\Omega_1^\varepsilon\setminus\Omega_{1/2}^\varepsilon)}\leq C_1e^{-\frac{C_2}{\sqrt\varepsilon+d(x')}}\|v\|_{L^p(\Omega_{1}^\varepsilon\setminus\Omega_{1/2}^\varepsilon)}.
\end{align}
The desired result is proved.

{\bf Case 2}: $h_1$ and $h_2$ are two $C^{1,\gamma}$ functions satisfying \eqref{assump-h-gamma}. In this case, \eqref{est-v-Dv} and \eqref{ite-Dv-flat2zz} are replaced by, respectively, 
\begin{equation}\label{est-v-Dv-2}
\int_{\Omega_t^\varepsilon}|Dv|^p\leq\frac{C}{(s-t)^p}\int_{\Omega_s^\varepsilon\setminus\Omega_t^\varepsilon}|v|^p
\end{equation}
and
\begin{equation}\label{ite-Dv}
\int_{\Omega_t^\varepsilon}|Dv|^p\leq \tilde C_0\left(\frac{\varepsilon+s^{1+\gamma}}{s-t}\right)^p\int_{\Omega_s^\varepsilon\setminus\Omega_t^\varepsilon}|Dv|^p.
\end{equation}
Take $t_0=r\in(\varepsilon^{1/(1+\gamma)},1/2)$ and $t_j=\big(1-\frac{1}{4}jr^{\gamma}\big)r$ with $j\in\mathbb N$ such that $j\leq 4r^{-\gamma}$. Now we take $s=t_j$ and $t=t_{j+1}$ in \eqref{ite-Dv}, we obtain
\begin{equation*}
\int_{\Omega_{t_{j+1}}^\varepsilon}|Dv|^p\leq 8^p\tilde C_0\int_{\Omega_{t_j}^\varepsilon\setminus\Omega_{t_{j+1}}^\varepsilon}|Dv|^p,
\end{equation*}
and thus,
\begin{equation*}
\int_{\Omega_{t_{j+1}}^\varepsilon}|Dv|^p\leq \frac{8^p\tilde C_0}{1+8^p\tilde C_0}\int_{\Omega_{t_j}^\varepsilon}|Dv|^p.
\end{equation*}
Choose $k=\lfloor\frac{1}{r^\gamma}\rfloor$ and iterate the above inequality $k$ times. By using \eqref{est-v-Dv-2} and the fact that $r<1/2$,  we have 
\begin{equation}\label{Dv-gamma}
\int_{\Omega_{3r/4}^\varepsilon}|Dv|^p\leq \left(\frac{8^p\tilde C_0}{1+8^p\tilde C_0}\right)^k\int_{\Omega_{r}^\varepsilon}|Dv|^p\leq C\mu_2^{\frac{1}{r^\gamma}}\int_{\Omega_1^\varepsilon\setminus\Omega_{1/2}^\varepsilon}|v|^p,
\end{equation}
where $\mu_2\in(0,1)$ and $C>0$ are constants depending on $n,p,c_3$, and $c_4$. For any $x\in\overline{\Omega_{1/3}^\varepsilon}$ and small $\varepsilon\in(0,100^{-2})$, we choose $r=\frac{17}{12}(\varepsilon^{1/(1+\gamma)}+|x'|)$ and $\rho=\frac{1}{20}(\varepsilon+|x'|^{1+\gamma})$, such that
$\Omega_{\rho}^\varepsilon(x)\subset\Omega_{3r/4}^\varepsilon$.
Similar to \eqref{est-Dv-p}, by applying classical estimates for the $p$-Laplace equation in $\Omega_{\rho}^\varepsilon(x)$ with $x\in\overline{\Omega_{1/3}^\varepsilon}$, and using \eqref{Dv-gamma}, we obtain 
\begin{equation*}
|Dv(x)|\leq C_1e^{-\frac{C_2}{\varepsilon^{\gamma/(1+\gamma)}+|x'|^\gamma}}\|v\|_{L^p(\Omega_{1}^\varepsilon\setminus\Omega_{1/2}^\varepsilon)},
\end{equation*}
where $C_1>0$ is a constant depending on $n,p,\gamma,c_3$, and $c_4$, and $C_2>0$ depends on $n,p,c_3$, and $c_4$. Therefore, the proof is finished. 
\end{proof}

%{\bf Data Availibility}

%Data sharing is not applicable to this article as no data sets were generated or analysed during the current study.

%{\bf Conflict of interest}

%The authors have no conflict of interest to declare.

%%%%%%%%%%%%%%%%%%%%%%%%%%%%%%%%%%%%%%%
\bibliographystyle{abbrv}
\bibliography{ref}

\end{document}